\documentclass{amsart}
\usepackage{amsmath,amssymb,amsfonts,amsthm,graphicx}
\usepackage{verbatim}
\usepackage[all]{xy}
\usepackage{amsthm}
\theoremstyle{plain}\newtheorem{thm}{Theorem}[section]\newtheorem{cor}[thm]{Corollary}
\theoremstyle{plain}\newtheorem{lem}[thm]{Lemma}
\theoremstyle{plain}\newtheorem{prop}[thm]{Proposition}
\theoremstyle{definition}\newtheorem{defn}[thm]{Definition}
\theoremstyle{remark}
\theoremstyle{remark}\newtheorem{exam}[thm]{Example}
\theoremstyle{remark}\newtheorem{rem}[thm]{Remark}
\theoremstyle{remark}
\def\mapr#1{{\mathop{\longrightarrow}\limits^{#1}}}
\def\gf#1,#2{{\mathop{\genfrac(){0cm}{0}{#1}{#2}}}} 
\let\sse=\subseteq\let\ssm=\smallsetminus
\let\mc=\mathcal
\let\mr=\mathrm


\author{Eric Emtander {$\dagger$}}
\address{$\dagger$ Department of Mathematics, Stockholm University, 
SE-106 91 Stockholm, Sweden}
\email{erice@math.su.se}

\author{Fatemeh Mohammadi {$\dagger\dagger$}} 
\address{$\dagger\dagger$ Faculty of Mathematics and Computer Sciences, 
Amirkabir University of Technology 424, Hafez Ave., P. O. Box 15875-4413, 
Tehran, Iran}
\email{f$_{-}$mohammadi@aut.ac.ir}

\author{Somayeh Moradi {$\ddagger$}} 
\address{$\ddagger$ Faculty of Mathematics and Computer Sciences, 
Amirkabir University of Technology 424, Hafez Ave., P. O. Box 
15875-4413, Tehran, Iran}
\email{s$_{-}$moradi@aut.ac.ir}

\begin{document}
\title{Some algebraic properties of hypergraphs}
\maketitle

\begin{abstract}
We consider Stanley--Reisner rings $k[x_1,\ldots,x_n]/I(\mc{H})$ where $I(\mc{H})$ is 
the edge ideal associated to some particular classes of hypergraphs. For instance, we consider 
hypergraphs that are natural generalizations of graphs that are lines and cycles, and for 
these we compute the Betti numbers. We also generalize upon some known results about chordal 
graphs and study a weak form of shellability.
\end{abstract}

\thanks{This paper was written during the visit of the second and the third
authors to Department of Mathematics, Stockholm University, Sweden. They
would like to express their deep gratitude to Professors J\"orgen Backelin and 
Ralf Fr\"oberg for their warm hospitality, support, and guidance. 
They also wish to thank the ministry of science, research and technology of Iran for the 
financial support.}

\section{Introduction}
Stanley--Reisner rings associated to graphs have been widely studied since edge ideals 
were introduced by R.~Villarreal in \cite{Vi}. Chordal graphs have been in particular 
focus and indeed, many natural results and questions seem to be connected to this class of 
graphs. See for example \cite{Eng,F1,VT,He2,He3,VT4}. More recently, also edge ideals of hypergraphs 
have been studied and many results and familiar properties of graphs have got hypergraph 
analouges. See for example \cite{E1,E2,VT,Vi2}. Also \cite{Ja} should be mentioned. In this paper 
the author computes Betti numbers of many classes of graphs.

Graphs may be generalized in several different ways of which hypergraphs is merely one. S.~Faridi 
introduced (\cite{Fa1}) a way of viewing a simplicial complex as a generalization of a graph. The 
complexes considered are called facet complexes. The idea is to consider the facets of a complex as 
a kind of generalized edges. Since a graph may be considered as a one dimensional simplicial complex, 
in this way one indeed obtains a natural generalization of graphs. Many nice results may be found in 
\cite{Fa1, Fa2, He2, He3, Zh}. In \cite{Zh}, the author introduces the class of quasi-forests, which 
relates to chordal graph, see \cite{He3}. Also, in this paper, we show that they are closely related to 
chordal hypergraphs.\\

In Section \ref{lines/cycles} we consider natural hypergraph generalizations of graphs that are lines 
and cycles. We call them line hypergraphs and hypercycles. Here we generalize upon known results about 
Betti numbers from \cite{Ja}.\\

In Section \ref{chordal} we consider chordal hypergraphs, as defined in \cite{E2}. Corollary \ref{graph} 
is known from \cite{He3}, and Theorem \ref{hypergraph} provides a partial generalization of this result. \\

The results in section \ref{connectivity} are related to the concept of connectivity. Connectivity in 
graph theory is a well known concept and we explore it 
in a hypergraph context. Using (hypergraph)connectivity we are able to generalize some results on bounds on 
Betti numbers from \cite{Ja}. Also, some results connecting the depth of certain Stanley-reisner rings 
to connectivity, given in \cite{F1}, can be seen in a wider context and explained a bit deeper.\\ 

In the last section, Section \ref{d-shellability}, we consider the notion of $d$-shellability, a weaker 
notion than ordinary shellability. We see that, to some extent, $d$-shellability behaves like ordinary 
shellability. In particular, we show that there is an Alexander dual notion called $d$-quotients, that 
provides a natural generalization of the concept of linear quotients that is known to be Alexander dual 
to shellability. For ideals with $d$-quotients we give a formula for the Betti numbers and show that 
these ideals admit Betti splittings.

\section{Basics}
We give some basics that will be used in the paper. A good reference for hypergraphs 
is Berge's book \cite{Be}.\\

Let ${\mc X}$ be a finite set and ${\mc E}=\{E_1,\ldots,E_s\}$ a
finite collection of non empty subsets of ${\mc X}$. The pair ${\mc{
H=(X,E)}}$ is called a {\bf hypergraph}. The elements of ${\mc X}$
and ${\mc E}$, respectively, are called the {\bf vertices} and the
{\bf edges}, respectively, of the hypergraph. If we want to specify
what hypergraph we consider, we may write $\mc{X(H)}$ and
$\mc{E(H)}$ for the vertices and edges respectively. A hypergraph is
called {\bf simple} if: (1) $|E_i|\ge2$ for all $i=1,\ldots,s$ and
(2) $E_j\sse E_i$ only if $i=j$. If the cardinality of ${\mc X}$ is
$n$ we often just use the set $[n]=\{1,2,\ldots,n\}$ instead of
$\mc{X}$.\\

Let $\mc H$ be a hypergraph. A {\bf subhypergraph} $\mc K$ of $\mc
H$ is a hypergraph such that $\mc{X}(\mc{K})\sse\mc{X(H)}$, and
$\mc{E(K)}\sse\mc{E(H)}$. If $\mc Y\sse\mc{X}$, the {\bf induced
hypergraph on} $\mc Y$, $\mc{H}_{\mc{Y}}$, is the subhypergraph with
$\mc{X(H_Y)}=\mc{Y}$ and with $\mc{E(H_Y)}$ consisting of the edges
of $\mc{H}$ that lie entirely in $\mc{Y}$. A hypergraph $\mc H$ is
said to be {\bf $d$-uniform} if $|E_i|=d$ for every edge $E_i\in
\mc{E( H)}$. By a uniform hypergraph we mean a hypergraph that is 
$d$-uniform for some $d$. Note that a simple $2$-uniform hypergraph is just an
ordinary simple graph. In this paper we consider only simple uniform hypergraphs, 
and hence, by hypergraph we will always mean a simple hypergraph.

A {\bf free vertex $v$} of a hypergraph is, if there is one, a vertex $v$ that lies in 
at most one edge.  

One type of hypergraphs of particular importance is the $d$-complete hypergraphs $K_n^d$:
The {\bf $d$-complete hypergraph}, $K_n^d$, on a set of $n$ vertices, is
defined by
\[
\mc{E}(K_n^d)=\gf{[n]},{d}
\]
where $\gf{F},d$ denotes the set of all subsets of $F$, of
cardinality $d$. If $n<d$, we interpret $K_n^d$ as $n$ isolated
points.\\

Recall that an {\bf (abstract) simplicial complex} on vertex set $[n]$ is a collection, 
$\Delta$, of subsets of $[n]$ with the property that $G\sse F,\,F\in\Delta \Rightarrow G\in\Delta$. 
The elements of $\Delta$ are called the {\bf faces} of the complex and the maximal (under inclusion) 
faces are called {\bf facets}. The set of facets of $\Delta$ we denote by $\mc{F}(\Delta)$. 
The {\bf dimension, $\dim F$, of a face} $F$ in $\Delta$, is defined to be $|F|-1$, and the dimension 
of $\Delta$ is defined as $\dim\Delta=\max\{\dim F ;\, F\in \Delta\}$. Note that the empty set 
$\emptyset$ is the unique $-1$ dimensional face of every complex that is not the void complex 
$\{\}$ which has no faces. The dimension of the void complex may be defined as $-\infty$. Let 
$V\sse [n]$. We denote by $\Delta_V$ the simplicial complex
\[
\Delta_V=\{F\sse [n]\,;\,F\in\Delta, F\sse V\}.
\]
$\tilde{H}_n(\Delta;k)$ will denote the reduced homology of (the chain complex of) $\Delta$ with 
coefficients in the field $k$.\\

To every hypergraph on vertex set $[n]$ we associate two simplicial complexes, the 
{\bf Independence complex of $\mc{H}$}, $\Delta_{\mc{H}}$, and the {\bf Clique complex of $\mc{H}$}, 
$\Delta(\mc{H})$. These are defined as follows: 
\[
\Delta_{\mc{H}}=\{F\sse [n] ; E\not\sse F,\, \forall E\in \mc{E(H)}\}
\]
\[
\Delta(\mc{H})=\{F\sse [n] ; \genfrac(){0cm}{0}Fd \sse \mc{E}(\mc{H})\}.
\]

Throughout the paper $R$ will denote some polynomial ring $k[x_1,\ldots,x_n]$. The number $n$ will be 
the cardinality of the vertex set of some hypergraph considered. We use the convention that for a subset 
$F\sse [n]$, $x^F=\prod_{i\in F}x_i$. Now, let $\Delta$ be a simplicial complex 
on $[n]$. The {\bf Stanley--Reisner ring} $R/I_\Delta$ of $\Delta$ is the quotient of the ring 
$R=k[x_1,\ldots,x_n]$ by the {\bf Stanley--Reisner ideal} 
\[
I_\Delta=(x^F ;\, F\not\in\Delta)
\]
generated by the nonfaces of $\Delta$. Conversely, to every squarefree monomial ideal $I$ one may 
associate a unique simplicial complex $\Delta_I$ in such way that its Stanley--Reisner ideal is 
precisely $I$. If $\Delta$ is a given simplicial complex on vertex set $[n]$, its {\bf Alexander dual 
simplicial complex} is defined by
\[
\Delta^\ast=\{F\sse [n]\,;\, [n]\ssm F\not\in\Delta\}.
\]
This yields a natural duality of squarefree monomial ideals as well. Hence we may denote the 
Stanley--Reisner ideal of $\Delta^\ast$ by $I^\ast_\Delta$. More information about the relations 
between the ideals $I_\Delta$ and $I^\ast_\Delta$ can be found in \cite{MS}

If $I$ is any monomial ideal, we denote by $\mc{G}(I)$ it unique set of minimal monomial generators.

If $\mc{H}$ is a hypergraph, the Stanley--Reisner ideal of $\Delta_{\mc{H}}$ is called the edge ideal of $\mc{H}$, 
and is denoted $I(\mc{H})$.\\

Recall that the projective dimension, $\mr{pd}_R(M)$ of an $R$-module $M$, is defined as 
\[
\mr{pd}_R(M)=\max\{i;\,\mr{Tor}_i^R(M,k)\neq0\}.
\]
This number may depend on the characteristic of $k$. Furthermore, recall that the number 
$\beta_i(M)=\dim_k\mr{Tor}_i^R(M,k)$ is called the $i$'th Betti number of $M$. Note that 
the $\mr{Tor}$-modules and the Betti numbers are, in the cases we consider, naturally 
$\mathbb{N}^n$-graded. For details about the algebra used in connection to simplicial complexes, 
we refer the reader to the books \cite{BH} and \cite{MS}.

\section{Hypercycles and line hypergraphs}\label{lines/cycles}
In \cite{Ja} the author computes the Betti numbers of, among other things, 
graphs that are lines and cycles. When trying to lift these concepts to 
hypergraph analouges, one may handle the increased degree of freedom in 
potentially many different ways. We define line hypergraphs $L_{n}^{d,\alpha}$, and 
hypercycles $C_n^{d,\alpha}$ in a natural way and we compute their Betti numbers. 
In particular, we see that the formulas for the Betti numbers are independent of 
the characteristic of the field $k$.

\begin{defn}
For integers $n,\alpha$ and $d$, we define the {\bf line hypergraph} $L_{n}^{d,\alpha}$ 
as the $d$-uniform hypergraph with edge set
$\mc{E}(L_{n}^{d,\alpha})=\{E_1,\ldots, E_n\}$ and vertex set
$\mc{X}(L_{n}^{d,\alpha})=\bigcup_{i=1}^n E_i$ such that:
\begin{itemize}
\item[$(i)$]{
For any $i\neq j$, $E_i\bigcap E_{j}\neq \emptyset $ if
and only if $j=i-1$ or $j=i+1$, $0\le i\le n$.}
\item[$(ii)$]{
$|E_i\bigcap E_{i+1}|=\alpha$ for all $i$, $1\leq i\leq n-1$.}
\end{itemize}
The length of a line hypergraph is defined as the number of edges.
\end{defn}

\begin{defn}
The {\bf hypercycle $C_{n}^{d,\alpha}$} is the $d$-uniform hypergraph with
edge set $\mc{E}(C_{n}^{d,\alpha})=\{E_1,\ldots, E_n\}$ and vertex set
$\mc{X}(C_{n}^{d,\alpha})=\bigcup_{i=1}^n E_i$ such that:
\begin{itemize}
\item[$(i)$]{
For any $i\neq j$ we have $|E_i\bigcap E_{j}|\neq \emptyset $
if and only if $j\equiv i-1,i+1$ $\mr{mod}(n)$.}
\item[$(ii)$]{
$|E_i\bigcap E_{i+1}|=\alpha$ for all $i$, $1\leq i\leq n-1$
and $|E_1\bigcap E_{n}|=\alpha$.}
\end{itemize}
\end{defn}

\subsection{Betti Numbers of Line hypergraphs}

Assume that in a $d$-uniform hypergraph $\mc{H}$ the following holds: 1) If $E_i$, $E_j$ are two 
intersecting edges, then $|E_i\cap E_j|=\alpha$ for some fixed $\alpha$. 2) Every edge has a free 
vertex. For such hypergraph it is easy to see that $2\alpha<d$ must hold. This is the motivation of 
the following theorem, and the reason that it is natural to divide our further considerations in two 
cases, namely the case when $2\alpha<d$ and the case when $2\alpha=d$.

\begin{thm}\label{betti}
Let $\mathcal{H}$ be a hypergraph such that each edge of
$\mathcal{H}$ has a free vertex. Then
$\beta_i(R/I(\mathcal{H}))=\genfrac(){0cm}{0}ni$.
\end{thm}

\begin{proof}
Let $\mc{E}(\mathcal{H})=\{E_1,\ldots,E_n\}$ and $v_i\in E_i$ be a free
vertex for any $i$, $1\leq i\leq n$. Then, for any $j$, $1\leq j\leq
n$, $x^{E_j}$ does not divide ${\rm
lcm}(x^{E_1},\ldots,x^{E_{j-1}},x^{E_{j+1}},\ldots,x^{E_n})$, since $v_j\nmid {\rm
lcm}(x^{E_1},\ldots,x^{E_{j-1}},x^{E_{j+1}},\ldots,x^{E_n})$. Therefore the Taylor
resolution of $R/I(\mathcal{H})$ is minimal and
$\beta_i(R/I(\mathcal{H}))=\genfrac(){0cm}{0}ni$.
\end{proof}

In the following theorem, we give a combinatorial interpretation of the
graded Betti numbers of a hypergraph in which all edges have a free
vertex.

\begin{thm}\label{u}
Let $\mathcal{H}$ be a hypergraph with edges $E_1,\ldots, E_n$ such
that each edge has a free vertex. Then
\[
\beta_{i,j}(R/I(\mathcal{H}))=\Big|\{\bigcup_{t=1}^i E_{l_t}; \ \
1\leq l_1<\cdots <l_i\leq n, |\bigcup_{t=1}^i E_{l_t}|=j\}\Big|.
\]
\end{thm}

\begin{proof}
Since each edge of $\mathcal{H}$ has a free vertex, it is enough
to find the number of basis elements $e_{k_1,\ldots,k_i}$ of degree
$j$ in the Taylor resolution of $R/I(\mathcal{H})$. We have
$\deg(e_{k_1,\ldots,k_i})=\deg({\rm lcm}(x^{E_{k_1}},\ldots,x^{E_{k_i}}))$. 
Since all $f_{k_l}$ are squarefree, 
$\deg({\rm lcm}(x^{E_{k_1}},\ldots,x^{E_{k_i}}))=|\bigcup_{l=1}^iE_{k_l}|$, which
completes the proof.
\end{proof}

We have the following corollary:

\begin{cor}\label{b1}
Let $L_n^{d,\alpha}$ be a line hypergraph such that $d>2\alpha$. Then $$
\beta_{n,j}(R/I(L_n^{d,\alpha}))= \left\lbrace
\begin{array}{c l}
0 & \text{if $j\neq n(d-\alpha)+\alpha$}\\
1 & \text{if $j=n(d-\alpha)+\alpha$}
\end{array}
\right.$$
\end{cor}

\begin{proof}
Since $d>2\alpha$, each edge has a free vertex. Thus by Theorem
\ref{u},
\[
\beta_{n,j}(R/I(L_n^{d,\alpha}))=\Big|\{\bigcup_{t=1}^n
E_{l_t}; \ 1\leq l_1<\cdots <l_n\leq n, \ |\bigcup_{t=1}^n
E_{l_t}|=j\}\Big|.
\]
Therefore it equals $1$ if $j=|\bigcup_{i=1}^n E_i|$ and $0$ if 
$j\neq |\bigcup_{i=1}^n E_i|$. Since $|\bigcup_{i=1}^n E_i|
=n(d-\alpha)+\alpha$, the assertion holds.
\end{proof}

Let $s_i$ $(1\leq i\leq r)$ be positive integers and $L_n^{d,\alpha}$ be 
a line hypergraph of length $n$. Set $E(s_1,\ldots,s_r,n)=\{\mc{H}\,;\,\ 
\mc{H} $ is a subhypergraph of $L_n^{d,\alpha}$, which is comprised of $r$ disjoint
line hypergraphs of lengths $s_1,\ldots,s_r$ with no isolated vertex\}. Then we 
have the following lemma:

\begin{lem}\label{l}
Let $s_1,\ldots,s_r$ be positive integers such that
$1\le s_1=\cdots=s_{l_1}<s_{l_1+1}=\cdots=s_{l_1+l_2}<s_{l_1+l_2+1}=\cdots=
s_{l_1+l_2+l_3}<\cdots<s_{l_1+\cdots+l_{t-1}+1}=\cdots=
s_{l_1+\cdots+l_t}=s_r$ and $s_1+\cdots+s_r=i$, then 
$|E(s_1,\ldots,s_r,n)|=\frac{r!}{l_1!\cdots l_t!}\genfrac(){0cm}{0}{n-i+1}{r}$.
\end{lem}

\begin{proof}
Let $L_n^{d,\alpha}$ be a line hypergraph of length $n$ and $S$ be the set
of hypergraphs $\mc{H}\in E(s_1,\ldots,s_r,n)$ such that $\mc{H}$ is comprised of
line hypergraphs $Q_1,\ldots,Q_r$ such that the length of $Q_i$ is $s_i$ 
and for any $1\leq i<j\leq r$, if $v_\lambda\in \mc{X}(Q_i)$ and $v_\gamma\in
\mc{X}(Q_j)$, then $\lambda<\gamma$. It can be seen that
$|E(s_1,\ldots,s_r,n)|=\frac{r!}{l_1!\cdots l_t!}|S|$. We claim that
there is a bijection between $S$ and the set
$\{(t_1,\ldots,t_{r+1}),\ t_1,t_{r+1}\geq 0,t_2,\ldots,t_{r}\geq 1,\
\sum_{l=1}^{r+1}t_l=n-i\}$. For any $\mc{H}\in S$, let $\mc{H}'$ be the subhypergraph
of $L_n^{d,\alpha}$ with edge set $\mc{E}(L_n^{d,\alpha})\setminus\mc{E}(\mc{H})$. Assume 
$\mc{H}'$ is comprised of $l$ line hypergraphs $Q'_1,\ldots,Q'_l$ such that for any $1\leq
i<j\leq l$, if $v_\lambda\in \mc{X}(Q'_i)$ and $v_\gamma\in\mc{X}(Q'_j)$,
then $\lambda<\gamma$. We have $r-1\leq l\leq r+1$. Set
$v_{\mc{H}}=(t_1,\ldots,t_{r+1})$, where $t_i(1\leq i\leq r+1)$ are as
follows:\\

\begin{itemize}
\item[$(i)$]If $l=r-1$, set $t_1,t_{r+1}=0$ and $t_i=l(Q'_{i-1})$ for any
$i,(2\leq i\leq r)$.
\item[$(ii)$]If $l=r$ and $v_0\in \mc{X}(Q_1)$, set $t_1=0$ and $t_i=l(Q'_i)$
for any
$i,(2\leq i\leq r+1)$.
\item[$(iii)$]If $l=r$ and $v_0\in \mc{X}(Q'_1)$, set $t_{r+1}=0$ and
$t_i=l(Q'_i)$ for any $i,(1\leq i\leq r)$.
\item[$(iv)$]If $l=r+1$, set $t_i=l(Q'_i)$ for any
$i,(1\leq i\leq r+1)$.
\end{itemize}

We have $\sum_{k=1}^{r+1}t_k=\sum_{k=1}^{l} l(Q'_k)=n-i$. Define a function
\[
\phi:S\longrightarrow \{(t_1,\ldots,t_{r+1}),\ t_1,t_{r+1}\geq
0,t_2,\ldots,t_{r}\geq 1,\ \sum_{l=1}^{r+1}t_l=n-i\}
\] 
by $\phi(\mc{H})=v_{\mc{H}}$. It is easy to see that $\phi$ is a bijection and hence 
\[
|S|=\Big|\{(t_1,\ldots,t_{r+1}),\ t_1,t_{r+1}\geq
0,t_2,\ldots,t_{r}\geq 1,\
\sum_{l=1}^{r+1}t_l=n-i\}\Big|
\]
\[
=\Big|\{(t'_1,\ldots,t'_{r+1}),\
t'_1,\ldots,t'_{r+1}\geq 1,\
\sum_{l=1}^{r+1}t'_l=n-i+2\}\Big|=\genfrac(){0cm}{0}{n-i+1}r.
\] 
Thus $|E(s_1,\ldots,s_r,n)|=\frac{r!}{l_1!\cdots l_t!}\genfrac(){0cm}{0}{n-i+1}r$.
\end{proof}

We now give the graded Betti numbers of a line hypergraphs
$L_n^{d,\alpha}$ in the case when $d>2\alpha$.

\begin{thm}\label{P}
Let $i<n$ be an integer and let $L_n^{d,\alpha}$ be a line hypergraph such
that $d>2\alpha$. Then
$\beta_{i,id-\alpha(i-r)}(R/I(L_n^{d,\alpha}))=
\genfrac(){0cm}{0}{i-1}{r-1}\genfrac(){0cm}{0}{n-i+1}r$ for any $r$, $1\leq r\leq i$ and
$\beta_{i,j}(R/I(L_n^{d,\alpha}))=0$ for all other $j$.
\end{thm}

\begin{proof}
Let $i<n$ and $j$ be integers such that
$\beta_{i,j}(R/I(L_n^{d,\alpha}))\neq 0$. Since $d>2\alpha$,
each edge has a free vertex. Thus as was shown in Theorem \ref{u},
$\beta_{i,j}(R/I(L_n^{d,\alpha})=|\{\bigcup_{t=1}^i E_{l_t};
\ \ |\bigcup_{t=1}^i E_{l_t}|=j\}|$. Let $E_{l_1},\ldots, E_{l_i}$
be some edges of $L_n^{d,\alpha}$ such that $|\bigcup_{t=1}^i
E_{l_t}|=j$ and let $\mathcal{H}$ be the subhypergraph of
$L_n^{d,\alpha}$ with edge set $\{E_{l_1},\ldots,E_{l_i}\}$ and
assume that $\mathcal{H}$ is comprised of $r$ line hypergraphs which are
of lengths $s_1,\ldots,s_r$. Then $s_1+\cdots+s_r=i$. Let
$L_t^{d,\alpha}\subseteq \mathcal{H}$ be the line hypergraph of length $s_t$, 
so that $|\mc{X}(L_t^{d,\alpha})|=s_td-\alpha(s_t-1)$. Therefore
$|\mc{X}(\mathcal{H})|=\sum_{t=1}^r (s_td-\alpha(s_t-1))=id-\alpha(i-r)$ 
and $j=id-\alpha(i-r)$ for some $1\leq r\leq i$. Hence
\[
\beta_{i,id-\alpha(i-r)}(R/I(L_n^{d,\alpha}))=\sum_{1\leq
s_1\leq \cdots\leq s_r, s_1+\cdots+s_r=i}|E(s_1,\ldots,s_r,n)|
\]
and $\beta_{i,j}(R/I(L_n^{d,\alpha}))=0$ for those $j$ that can
not be written in the form $id-\alpha(i-r)$ for some $r$. Let
$l_1,\ldots,l_m$ be positive integers and $P_{(l_1,\ldots,l_m)}$ be
the number of solutions of $x_1+\cdots+x_r=i$ such that $x_j\geq
1 (1\leq j\leq r)$ and $l_i$ components of $(x_1,\ldots,x_r)$ are
equal for all $i$, $1\leq i\leq m$. Then $\sum_{m\geq 1,l_i\geq
1}P_{(l_1,\ldots,l_m)}=\genfrac(){0cm}{0}{i-1}{r-1}$. Also the number of
solutions of $x_1+\cdots+x_r=i$ such that $1\leq
x_1=\cdots=x_{l_1}<x_{l_1+1}=\cdots=x_{l_1+l_2}<x_{l_1+l_2+1}=\cdots=
x_{l_1+l_2+l_3}<\cdots<x_{l_1+\cdots+l_{m-1}+1}=\cdots=
x_{l_1+\cdots+l_m}=x_r$ is equal to $P_{(l_1,\ldots,l_m)}
\frac{l_1!\cdots l_m!}{r!}$. Thus using Lemma \ref{l} we see that
\[
\beta_{i,id-\alpha(i-r)}(R/I(L_n^{d,\alpha}))=\sum_{m\geq 1,l_i\geq 1}
P_{(l_1,\ldots,l_m)} \frac{l_1!\cdots l_m!}{r!}\\
\frac{r!}{l_1!\cdots l_m!}\genfrac(){0cm}{0}{n-i+1}r
\]
\[
=\sum_{m\geq 1,l_i\geq 1}
P_{(l_1,\ldots,l_m)} \genfrac(){0cm}{0}{n-i+1}r=\genfrac(){0cm}{0}{i-1}{r-1}
\genfrac(){0cm}{0}{n-i+1}r.
\]
The proof is complete.
\end{proof}

We now consider $L_n^{d,\alpha}$ in the case where $d=2\alpha$.

\begin{thm}\label{PI}
Let $L_n^{d,\alpha}$ be a line hypergraph such that $d=2\alpha$. Then the
non zero Betti numbers of $R/I(L_n^{d,\alpha})$ are, for $2j\geq i$, 
as follows:
\[
\beta_{i,j\alpha}(R/I(L_n^{d,\alpha}))=\gf{j-i},{2i-j}
\gf{n+1-2j+2i},{j-i}+\gf{j-i-1},{2i-j}\gf{n+1-2j+2i},{j-i-1}.
\]
\end{thm}

\begin{proof}
Let $\mc{E}(L_n^{d,\alpha})=\{E_1,\ldots,E_n\}$, where
$E_i=\{x_{1,i},\ldots,x_{d,i}\}$. Set $X_i=x_{1,i}\cdots
x_{\alpha,i}$ and $X_{i+1}=x_{\alpha+1,i}\cdots x_{d,i}$ for any
$i$, $1\leq i\leq n$. Here $\{x_{\alpha+1,i}\cdots x_{d,i}\}$ are, for every 
$i=1,\ldots,n-1$, the vertices in the intersection $E_i\cap E_{i+1}$. 
Then $I(L_n^{d,\alpha})=(X_1X_2,\ldots,X_nX_{n+1})$. Since the $X_i$'s are independent 
variables and $\deg(X_i)=\alpha$, we have
$\beta_{i,j\alpha}(R/I(L_n^{d,\alpha}))=\beta_{i,j}(R/(X_1X_2,\ldots,
X_nX_{n+1}))$. The result now follows, using Theorem $7.7.34$ of \cite{Ja}.
\end{proof}

\subsection{Betti Numbers of Hypercycles}

We start by giving a corollary that is similar to Corollary \ref{b1}.

\begin{cor}\label{b}
Let $C_n^{d,\alpha}$ be a hypercycle such that each edge has a free
vertex. Then $ \beta_{n,j}(R/I(C_n^{d,\alpha}))= \left\lbrace
\begin{array}{c l}
0 & \text{if $j\neq n(d-\alpha)$}\\
1 & \text{if $j=n(d-\alpha)$}
\end{array}
\right.$
\end{cor}

\begin{proof}
By Theorem \ref{u},
$\beta_{n,j}(R/I(C_n^{d,\alpha}))=\Big|\{\bigcup_{t=1}^n
E_{l_t}; \ 1\leq l_1<\cdots<l_n\leq n \ |\bigcup_{t=1}^n
E_{l_t}|=j\}\Big|$. Therefore it is equal to $1$ if $j=|\bigcup_{t=1}^n
E_{t}|$ and $0$ if $j\neq |\bigcup_{t=1}^n E_{t}|$. Since
$|\bigcup_{t=1}^n E_{t}|=n(d-\alpha)$, the assertion holds.
\end{proof}

We compute the Betti numbers of the hypercycle $C_n^{d,\alpha}$ in the same manner 
as in the previous section. That is, we consider the two cases $d>2\alpha$ and 
$d=2\alpha$ separetely.\\

Let $s_i$ $(1\leq i\leq r)$ be integers and $C_n^{d,\alpha}$ be a cycle of length $n$. 
Set $F(s_1,\ldots,s_r,n)\\=\{H\,;\,\ H $ is a subhypergraph of $C_n^{d,\alpha}$, which is
comprised of $r$ disjoint line hypergraphs of lengths $s_1,\ldots,s_r$ and with no 
isolated vertex\}, then we have the following lemma:

\begin{lem}\label{k}
Let $s_1,\ldots,s_r$ be positive integers such that
$s_1=\cdots=s_{l_1}<s_{l_1+1}=\cdots=s_{l_1+l_2}<s_{l_1+l_2+1}=\cdots
=s_{l_1+l_2+l_3}<\cdots<s_{l_1+\cdots+l_{t-1}+1}=\cdots=s_{l_1+\cdots+l_t}
=s_r$ and $s_1+\cdots+s_r=i$, then $|F(s_1,\ldots,s_r,n)|=
\frac{n(r-1)!}{l_1!\cdots l_t!}\gf{n-i-1},{r-1}$.
\end{lem}

\begin{proof}
Let $E_1,\ldots,E_{s_r}\in\mc{E}(C_n^{d,\alpha})$. The number of subgraphs of
$C_n^{d,\alpha}$, which are comprised of $r$ line hypergraphs 
$Q_1,\ldots,Q_r$ of lengths $s_1,\ldots,s_r$ and $\mc{E}(Q_r)=\{E_1,\ldots,E_{s_r}\}$ 
is equal to the number of subhypergraphs of $L_{n-s_r-2}^{d,\alpha}$, which are comprised 
of $r-1$ line hypergraphs of lengths $s_1,\ldots,s_{r-1}$. Therefore
$|F(s_1,\ldots,s_r,n)|=n|E(s_1,\ldots,s_{r-1},n-s_r-2)|$, since the
line hypergraph $Q_r$ can start from any of the $n$ edges. By Lemma \ref{l},
$|E(s_1,\ldots,s_{r-1},n-s_r-2)|=\frac{(r-1)!}{l_1! \cdots
l_t!}\gf{n-i-1},{r-1}$, which completes the proof.
\end{proof}

We now give the Betti numbers of $C_n^{d,\alpha}$ in the case where $d>2\alpha$.

\begin{thm}\label{betti1}
Let $i<n$ be an integer and let $C_n^{d,\alpha}$ be a hypercycle
such that $d>2\alpha$. Then $\beta_{i,id-\alpha(i-r)}(R/I(C_n^{d,\alpha}))=\frac{n}{r}
\gf{i-1},{r-1}\gf{n-i-1},{r-1}$ for any $r$, $0\leq r\leq i$ and
$\beta_{i,j}(R/I(C_n^{d,\alpha}))=0$ for other $j$'s.
\end{thm}

\begin{proof}
Since $d>2\alpha$, each edge has a free vertex. Therefore as was
shown in Theorem \ref{u},
$\beta_{i,j}(R/I(C_n^{d,\alpha}))=|\{\bigcup_{t=1}^i E_{l_t};
\ \ |\bigcup_{t=1}^i E_{l_t}|=j\}|$. Let $E_{l_1},\ldots, E_{l_i}$
be some edges of $C_n^{d,\alpha}$ such that $|\bigcup_{t=1}^i
E_{l_t}|=j$. Let $\mathcal{H}$ be a  subhypergraph of
$C_n^{d,\alpha}$ with edges $\{E_{l_1},\ldots,E_{l_i}\}$ and assume
that $\mathcal{H}$ is comprised of $r$ line hypergraphs which are of
lengths $s_1,\ldots,s_r$. Then $s_1+\cdots+s_r=i$. Let $L_t^{d,\alpha}\subseteq
\mc{H}$ be a line hypergraph of length $s_t$, then, as in the proof of Theorem \ref{P},
$|\mc{X}(L_t^{d,\alpha})|=s_td-\alpha(s_t-1)$ and therefore
$|\mc{X}(\mathcal{H})|=\sum_{t=1}^r (s_td-\alpha(s_t-1))=id-\alpha(i-r)$.
Thus $j=id-\alpha(i-r)$ and we see that 
\[
\beta_{i,id-\alpha(i-r)}(R/I({C_{n}^{d,\alpha}}))=\sum_{1\leq
s_1\leq\cdots\leq s_r, s_1+\cdots+s_r=i}|F(s_1,\ldots,s_r,n)|
\]
and $\beta_{i,j}(R/I(C_n^{d,\alpha}))=0$ for all $j$ that can not
be written in the form $id-\alpha(i-r)$ for some $r$.

Construct the numbers $P_{(l_1,\ldots,l_m)}$ and $P_{(l_1,\ldots,l_m)}
\frac{l_1!\cdots l_m!}{r!}$ in the same way as in the proof of Theorem \ref{P}. 
Thus using Lemma \ref{k} we see that
\[
\beta_{i,id-\alpha(i-r)}(R/I({C_{n}^{d,\alpha}}))
=\sum_{m\geq 1,l_i\geq 1}P_{(l_1,\ldots,l_m)}\\ \frac{l_1!\cdots l_m!}{r!}
\frac{n(r-1)!}{l_1!\cdots l_m!}\gf{n-i-1},{r-1}=
\]
\[
=\frac{n}{r}\gf{i-1},{{r-1}}\gf{{n-i-1}},{{r-1}},
\]
and the proof is complete. 
\end{proof}

Now consider the case where $d=2\alpha$.

\begin{thm}\label{to}
Let $C_n^{d,\alpha}$ be a hypercycle such that $d=2\alpha$. Then the
non zero Betti numbers of $R/I(C_n^{d,\alpha})$ are all in degree 
$\alpha j$, where $j\leq n$ and are as follows:
\begin{itemize}
\item[$(i)$]{
If $j<n$ and $2i\geq j$, then $\beta_{i,\alpha
j}(R/I(C_n^{d,\alpha}))=\frac{n}{n-2(j-i)}\gf{j-i},{2i-j}\gf{n-2(j-i)},{j-i}$}
\item[$(ii)$]{
If $n\equiv 1$ mod $3$ , $\beta_{\frac{2n+1}{3},\alpha
n}(R/I(C_n^{d,\alpha}))=1$}
\item[$(iii)$]{
If $n\equiv 2$ mod $3$ , $\beta_{\frac{2n-1}{3},\alpha
n}(R/I(C_n^{d,\alpha}))=1$}
\item[$(iv)$]{If $n\equiv 0$ mod $3$ , $\beta_{\frac{2n}{3},\alpha
n}(R/I(C_n^{d,\alpha}))=2$.}
\end{itemize}
\end{thm}

\begin{proof}
Let $\mc{E}(C_n^{d,\alpha})=\{E_1,\ldots,E_n\}$, where
$E_i=\{x_{1,i},\ldots,x_{d,i}\}$. Set $X_i=x_{1,i}\cdots
x_{\alpha,i}$ and $X_{i+1}=x_{\alpha+1,i}\cdots x_{d,i}$ for any
$i$, $1\leq i\leq n-1$, where $X_i$ and $X_{i+1}$ denote the same things as they do in 
the proof of Theorem \ref{PI}. Then 
\[
I(C_n^{d,\alpha})=(X_1X_2,\ldots,
X_{n-1}X_n,X_nX_1).
\]
Since $X_i$ are independent variables and 
$\deg(X_i)=\alpha$, we have $\beta_{i,j\alpha}(R/I(C_n^{d,\alpha}))=
\beta_{i,j}(X_1X_2,\ldots,X_{n-1}X_n,X_nX_1)$. Using Theorem $7.6.28$ 
of \cite{Ja}, the result follows.
\end{proof}

We recall from \cite{Ja} that a star graph is a complete bipartite graph $K_{1,n}$ for some $n$. One 
generalization of this graph is the $d$-complete bipartite hypergraph $K_{1,n}^d$. This kind of
hypergraphs are considered in \cite{E1}. Another way of generalizing the star graph is to 
focus more on its appearance. Considering the following picture of $K_{1,4}$ (here on vertex set 
$\{y\}\sqcup \{x_1,x_2,x_3,x_4\}$):
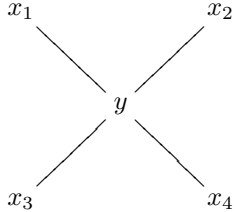
\[
\xymatrix{
x_1\ar@{-}[dr]&& x_2\ar@{-}[dl]\\
&y&\\
x_3\ar@{-}[ur]&& x_4\ar@{-}[ul]}
\]

it is motivated to say that the hypergraphs $\mc{H}$ considered in the next theorem also are
natural generalizations of the star graph.

\begin{thm}
Let $\mathcal{H}$ be a $d$-uniform hypergraph with edges
$E_{1},\ldots,E_{n}$ such that for any $i\neq j$, $|E_i\cap
E_j|=|\bigcap_{l=1}^n E_l|=\alpha$ and each edge has a free vertex.
Then $\beta_{i,j}(R/I(\mathcal{H}))\neq 0$ if and only if
$j=id-\alpha(i-1)$ and
$\beta_{i,id-\alpha(i-1)}(R/I(\mathcal{H}))=\gf n,i$.
\end{thm}
\begin{proof}
{For any $i$ and any edges $E_{l_1},\ldots,E_{l_i}$, we have
$|\bigcup_{t=1}^i E_{l_t}|=id-\alpha(i-1)$ and the number of
elements of the form $\bigcup_{t=1}^i E_{l_t}$ is $\gf n,i$.
Therefore using Theorem \ref{u}, we have
$\beta_{i,j}(R/I(\mathcal{H}))\neq 0$ if and only if
$j=id-\alpha(i-1)$ and
$\beta_{i,id-\alpha(i-1)}(R/I(\mathcal{H}))=\gf n,i$.}
\end{proof}

\section{Chordal hypergraphs}\label{chordal}

Chordal graphs have been condsidered more or less extensively for some time now. 
A core results in this area is the following theorem by R.~Fr\"oberg, \cite{F1}:

\begin{thm}
A graph $\mc{G}$ is chordal if and only if $R/I(\mc{G}^c)$ has linear resolution.
\end{thm}

From this result Fr\"oberg easily concluded that the complexes $\Delta(\mc{G})$, $\mc{G}$ 
chordal, that are also Cohen--Macaulay, are the ones in which $\mc{G}$ is a generalized 
$d$-tree. See \cite{F1} for details. More nice results associated to chordal graphs are 
given in \cite{He3} and \cite{VT4}. In \cite{E2} a notion of chordal hypergraph is given. 
We will use this to generalize upon some previously known results on chordal graphs.\\

The following definition is from \cite{E2}:

\begin{defn}
A {\bf chordal hypergraph} is a $d$-uniform hypergraph, obtained
inductively as follows:
\begin{itemize}
\item{$K_n^d$ is a chordal hypergraph, $n,d\in\mathbb{N}$.}
\item{If $\mc{G}$ is chordal, then so is
$\mc{H}$=$\mc{G}\bigcup_{K_j^d} K_i^d$, for $0\le j<i$. (We attach $K_i^d$ to $\mc{G}$ 
in a common (under identification) $K_j^d.$)}
\end{itemize}
\end{defn}

In connection to this definition we mention two facts: First, it is easy to see that 
every line hypergraph $L_n^{d,\alpha}$ is a chordal hypergraph. Indeed, it may be 
written as
\[
L_n^{d,\alpha}=K_d^d\cup_{K_\alpha^d}K_d^d\cup_{K_\alpha^d}K_d^d\cup_{K_\alpha^d}\cdots
\cup_{K_\alpha^d}K _d^d\cup_{K_\alpha^d}K_d^d.
\]
Also, as in chordal graphs, a chordal hypergraph does not contain any induced 
hypercycle $C_n^{d,\alpha}$. This is because an induced hypergraph of a chordal hypergraph 
is again chordal (see \cite{E2}). However, a hypercycle is not chordal.

In the following theorem we will use the well known fact that the Alexander dual notion of 
shellability, is the concept of linear quotients. The also follows from Theorem 
\ref{d-quot/d-shell}.

\begin{thm}\label{hypergraph}
If $\mc{H}$ is chordal hypergraph, then $I_{\Delta(\mc{H})}$ has linear quotients.
\end{thm}

\begin{proof}
Assume that $\mc{H}$ is chordal, then inductively according to the
definition of chordal hypergraph we show that $\Delta^*$ is pure shellable. 
We have $\Delta^*=\langle \mc{X(H)}\setminus E, |E|=d, E\notin \mc{E(H)}\rangle$.
If $\mc{H}=K_n ^d$, then $\Delta^*=\emptyset$ and the result holds.
Assume that $\mc{H}=\mc{L}\bigcup_{K_j ^d} K_i ^d$, where $\mc{L}$ is chordal. By
induction hypothesis we have $\Delta_L ^*=\langle \mc{X(L)}\setminus E,
|E|=d, E\notin \mc{E(L)}\rangle$ is pure shellable with ordering
$H_1<\cdots<H_n$. Let $E$ be a subset of $\mc{X(H)}$ such that $|E|=d$
and $E\notin \mc{E(H)}$. Let $S_1=\mc{X(L)}\setminus \mc{X}(K_j ^d)$ and 
$S_2=\mc{X}(K_i^d)\setminus \mc{X}(K_j ^d)$, then $|E\cap S_1|\geq 1$. If $|E\cap
S_2|=0$, then $E\subseteq \mc{X(L)}$ and $\mc{X(L)}\setminus E\in \Delta_{\mc{L}}
^*$. Let $S_1=\{x_1,\ldots,x_r\}$, $S_2=\{y_1,\ldots,y_s\}$ and
$S_3=\mc{X}(K_j ^d)=\{z_1,\ldots,z_j\}$. For any set $S$ and integer $i$,
set $\Delta_i=\Delta_S ^{[i-1]}$, where $\Delta_S$ is the full
simplex on $S$. Then we have $\Delta_i$ is shellable. Let $i\geq 1$
and $1\leq t\leq i$ be integers. Consider the simplicial complex
$A_{i,t}=\langle E^c;\ E\subseteq \mc{X(H)}, |E|=d$, $|E\cap S_2|=d-i$,
$|E\cap S_1|=t$ and $|E\cap S_3|=i-t\rangle$. Each element of
$A_{i,t}$ has the form $\{y_1,\ldots,y_{d-i}\}^c\cap
\{x_1,\ldots,x_t\}^c\cap \{z_1,\ldots,z_{i-t}\}^c$ for some
$x_i,y_j,z_k$. Consider the ordering on the elements
of $A_{i,t}$ as follows:\\
Let $A_1,\ldots, A_m$ be a shelling for the simplicial complex
$\langle F^c,\ |F|=d-i,\ F\subseteq \{y_1,\ldots,y_s\}\rangle$,
$B_1,\ldots, B_n$ be a shelling for the simplicial complex $\langle
F^c,\ |F|=t,\ F\subseteq \{x_1,\ldots,x_r\}\rangle$ and $C_1,\ldots,
C_k$ be a shelling for the simplicial complex $\langle F^c,\
|F|=i-t,\ F\subseteq \{z_1,\ldots,z_j\}\rangle$. Consider the ordering\\
$A_1\cap B_1\cap C_1< A_2\cap B_1\cap C_1<\cdots<A_m\cap B_1\cap
C_1<A_1\cap B_2\cap C_1<A_2\cap B_2\cap C_1<\cdots<A_m\cap B_2\cap
C_1<\cdots<A_1\cap B_n\cap C_1<A_2\cap B_n\cap C_1<\cdots<A_m\cap
B_n\cap C_1<\cdots<A_1\cap B_1\cap C_k<A_2\cap B_1\cap
C_k<\cdots<A_m\cap B_1\cap C_k<\cdots<A_1\cap B_n\cap C_k<A_2\cap
B_n\cap C_k<\cdots<A_m\cap B_n\cap C_k$ for the facets of $A_{i,t}$.
Let $A_i\cap B_j\cap C_k<A_{i'}\cap B_{j'}\cap C_{k'}$, then $k\leq
k'$. If $k<k'$, then there exists $v\in C_{k'}\setminus C_k $ and
$l<k'$ such that $C_{k'}\setminus C_l=\{v\}$. So $v=z_n$ for some
$n$ and $v\in A_{i'}\cap B_{j'}\cap C_{k'}\setminus A_i\cap B_j\cap
C_k$ and $A_{i'}\cap B_{j'}\cap C_{k'}\setminus A_{i'}\cap
B_{j'}\cap C_l=\{v\}$. Since $l<k'$, we have $A_{i'}\cap B_{j'}\cap
C_l<A_{i'}\cap B_{j'}\cap C_{k'}$. Let $k=k'$. Then $j\leq j'$. If
$j<j'$, by the same way there exists $v\in A_{i'}\cap B_{j'}\cap
C_k\setminus A_i\cap B_j\cap C_k$ and $l<j'$ such that $A_{i'}\cap
B_{j'}\cap C_k\setminus A_{i'}\cap B_l\cap C_k=\{v\}$. Therefore let
$k=k'$ and $j=j'$. Then $i<i'$. So there exists $v\in
A_{i'}\setminus A_i$ and $l<i'$ such that $A_{i'}\setminus
A_l=\{v\}$. Therefore $v=y_n$ for some $n$ and $v\in A_{i'}\cap
B_j\cap C_k\setminus A_i\cap B_j\cap C_k$ and $A_{i'}\cap B_j\cap
C_k\setminus A_l\cap B_j\cap C_k=\{v\}$ and $A_l\cap B_j\cap
C_k<A_{i'}\cap B_j\cap C_k$. So the above ordering is a shelling for
$A_{i,t}$. Now consider an ordering
for $\Delta^*$ as follows:\\
For $F\in A_{i,t}$ and $G\in A_{j,s}$, set $F<G$ if $i<j$ or $i=j$
and $t>s$. Also for any $G=H_j\cup S_2$ and $F\in A_{i,t}$, set $F<G$ 
and $H_1\cup S_2<\cdots<H_n\cup S_2$. 

Let $F$ and $G$ be two facets of $\Delta^*$ such that $F<G$ and let
$F$ be a facet of $A_{i,t}$ and $G$ be a facet of $A_{j,s}$. Then
$i\leq j$. The case $(i,t)=(j,s)$ is discussed above. Assume that
$(i,t)\neq (j,s)$ and $i=j$. Then $t>s$.

Let $F=\{y_{k_1},\ldots,y_{k_{d-i}}\}^c\cap
\{x_{k'_1},\ldots,x_{k'_t}\}^c\cap
\{z_{k''_1},\ldots,z_{k''_{i-t}}\}^c$ and
$G=\{y_{l_1},\ldots,\\y_{l_{d-i}}\}^c\cap
\{x_{l'_1},\ldots,x_{l'_s}\}^c\cap
\{z_{l''_1},\ldots,z_{l''_{i-s}}\}^c$. Since $t>s$, $i-t<i-s$. Also
there exists $x_{k'_\lambda}\notin \{x_{l'_1},\ldots,x_{l'_s}\}$ for
some $1\leq \lambda\leq t$. So $x_{k'_\lambda}\in G\setminus F$ and
$G\setminus (\{y_{l_1},\ldots,y_{l_{d-i}}\}^c\cap
\{x_{l'_1},\ldots,x_{l'_s},x_{k'_\lambda}\}^c\cap
\{z_{l''_1},\ldots,z_{l''_{i-s-1}}\}^c)=
\{y_{l_1},\ldots,y_{l_{d-i}},x_{l'_1},
\ldots,x_{l'_s},x_{k'_\lambda},z_{l''_1}, \ldots,\\z_{l''_{i-s-1}}\}
\setminus
\{y_{l_1},\ldots,y_{l_{d-i}},x_{l'_1},\ldots,x_{l'_s},z_{l''_1},\ldots,
z_{l''_{i-s}}\}=\{x_{k'_\lambda}\}$.
Also $\{y_{l_1},\ldots,y_{l_{d-i}}\}^c\cap
\{x_{l'_1},\ldots,x_{l'_s},x_{k'_\lambda}\}^c\cap
\{z_{l''_1},\ldots,z_{l''_{i-s-1}}\}^c<G$.

Now let $i<j$ and
$G=\{y_{l_1},\ldots,y_{l_{d-j}}\}^c\cap
\{x_{l'_1},\ldots,x_{l'_s}\}^c\cap
\{z_{l''_1},\ldots,z_{l''_{j-s}}\}^c$. Then there exists
$y_{k_{\lambda}}\in G\setminus F$. Since $j\geq 2$, we have $s\geq
2$ or $j-s\geq 1$. If $s\geq 2$, then $G\setminus
\{y_{l_1},\ldots,y_{l_{d-j}},y_{k_{\lambda}}\}^c\cap
\{x_{l'_1},\ldots,x_{l'_{s-1}}\}^c\cap
\{z_{l''_1},\ldots,z_{l''_{j-s}}\}^c=\{y_{k_{\lambda}}\}$. Since
$\{y_{l_1},\ldots,y_{l_{d-j}},y_{k_{\lambda}}\}^c\cap
\{x_{l'_1},\ldots,x_{l'_{s-1}}\}^c\cap
\{z_{l''_1},\ldots,z_{l''_{j-s}}\}^c\in A_{j-1,s-1}$, then in the
ordering it appears before $G$. If $j-s\geq 1$, then $G\setminus
\{y_{l_1},\ldots,y_{l_{d-j}},y_{k_{\lambda}}\}^c\cap
\{x_{l'_1},\ldots,x_{l'_s}\}^c\cap
\{z_{l''_1},\ldots,z_{l''_{j-s-1}}\}^c=\{y_{k_{\lambda}}\}$ and
$\{y_{l_1},\ldots,y_{l_{d-j}},y_{k_{\lambda}}\}^c\cap
\{x_{l'_1},\ldots,x_{l'_s}\}^c\\\cap
\{z_{l''_1},\ldots,z_{l''_{j-s-1}}\}^c\in A_{j-1,s}$.

Now let $F<G$, where $F=\{y_{k_1},\ldots,y_{k_{d-i}}\}^c\cap
\{x_{k'_1},\ldots,x_{k'_t}\}^c\cap
\{z_{k''_1},\ldots,z_{k''_{i-t}}\}^c$ and $G=H_j\cup
S_2=\{x_{l_1},\ldots,x_{l_{\lambda}},z_{l'_1},\ldots,z_{l'_{d-\lambda}}\}^c$.
Then there exists $y_m\in G\setminus F$. If $\lambda>1$, then
$G\setminus
\{y_m,x_{l_1},\ldots,x_{l_{\lambda-1}},z_{l'_1},\ldots,z_{l'_{d-\lambda}}\}^c=\{y_m\}$.
Otherwise $\lambda=1$ and $d-\lambda>1$. Then $G\setminus
\{y_m,x_{l_1},\ldots,x_{l_{\lambda}},z_{l'_1},\ldots,z_{l'_{d-\lambda-1}}\}^c=\{y_m\}$.
\end{proof}

\begin{cor}\label{graph}
The graph $\mc{G}$ is chordal if and only if $I_{\Delta(\mc{G})}$ has linear
quotients.
\end{cor}

\begin{proof}
The fact that the edge ideal of a chordal graph has linear quotients, follows from the theorem. 
Assume $R/I(\Delta(\mc{G}))$ has linear quotients. Then it has linear resolution and thus, 
$\mc{G}$ is chordal.
\end{proof}

In \cite{F1} Fr\"oberg considers a class of chordal graphs called $n$-trees.

\begin{defn}
A $n$-tree is a chordal graph defined inductively as follows:
\begin{itemize}
\item{$K_{n+1}$ is a $n$-tree.}
\item{If $\mc{G}$ is a $n$-tree, then so is
$\mc{H}$=$\mc{G}\bigcup_{K_{n}} K_{n+1}$. (We attach $K_{n+1}$ to $\mc{G}$ 
in a common (under identification) $K_{n}$)}
\end{itemize}
\end{defn}

Now, consider the corresponding subclass $\mc{T}_d$ of the class of chordal hypergraphs. That is, 
$\mc{T}_d$ is the class of chordal hypergraphs described as follows:
\begin{itemize}
\item{$K_{n+1}^d$ belongs to $\mc{T}_d$.}
\item{If $\mc{G}$ belongs to $\mc{T}_d$, then so does
$\mc{H}$=$\mc{G}\bigcup_{K_{n}^d} K_{n+1}^d$. (We attach $K_{n+1}^d$ to $\mc{G}$ 
in a common (under identification) $K_n^d.$)}
\end{itemize}

We get the following results:

\begin{thm}
For any hypergraph $\mc{H}$ in $\mc{T}_d$, the clique complex $\Delta(\mc{H})$ is pure 
shellable and hence Cohen--Macaulay.
\end{thm}

\begin{proof}
The proof is by induction. If $\mc{H}=K_{n+1}^d$, then $\Delta(\mc{H})$ is a simplex and 
pure shellable. Let $\mc{H}=\mc{G}\cup_{K_{n}^d}K_{n+1}^d$ and $F_1<\cdots<F_r$ be a shelling 
for $\Delta(\mc{G})$. Then $\Delta(\mc{H})=\langle F_1,\ldots,F_r,F_{r+1}\rangle$,
where $F_{r+1}=\mc{X}(K_{n+1}^d)$. Let $\mc{X}(K_{n}^d)=L$. Then $L\subseteq
F_i$ for some $1\leq i\leq r$. We claim that
$F_1<\cdots<F_r<F_{r+1}$ is a shelling for $\Delta(\mc{H})$. Let
$F_{r+1}=L\cup\{v\}$. Then for any $j\leq r$, one has $v\in
F_{r+1}\setminus F_j$ and $F_{r+1}\setminus F_i=\{v\}$.
\end{proof}

\begin{cor}
For any $d$-tree $\mc{G}$, the clique complex $\Delta(\mc{G})$ is pure shellable and hence
Cohen--Macaulay.
\end{cor}

In the proof of the following proposition, we will use the fact that the Stanley--Reisner 
ideal of the complex $\Delta_{I(K_n^d)}^\ast$ is shellable. One may in fact use a 
lexicographic shelling, so, by symmetry, one may start the shelling with any facet of 
the complex.

\begin{prop}
Let $\mc{H}=K_m^d\cup_{K_j^d}K_i^d$, $m\ge d$ be given. Then $I(\mc{H})$ has linear 
quotients precisely when 

\begin{itemize}
\item[$(i)$]{$i,j<d$, or}
\item[$(ii)$]{$i\geq d$, and $j=m-1$ or $j=i-1$.}
\end{itemize}
\end{prop}

\begin{proof}
Put $A=\mathcal{X}(\mathcal{H})\setminus \mathcal{X}(K_m^d)$ and
$B=\mathcal{X}(\mathcal{H})\setminus \mathcal{X}(K_i^d)$. We will
show that the Alexander dual complex of $\mathcal{H}$ is shellable
precisely in the cases mentioned above. Assume that
$F'_1<\cdots<F'_t$ is a shelling of $\Delta_{I^*(K_m^d)}$. Set
$F_i=F'_i\cup A$ for $i=1,\ldots,t$. In case $(i)$ the sequence
$F_1<\cdots<F_t$ is a shelling of $\Delta_{I^*(\mathcal{H})}$.

In case $(ii)$ we find a shelling when $i\geq d$ and $j=m-1$. 
The case $j=i-1$ is similar. Let $i\geq d$, $j=m-1$ and
$\mathcal{X}(K_m^d)\setminus \mathcal{X}(K_j^d)=\{v\}$. Let
$G'_1<\cdots<G'_s$ be a shelling of $\Delta_{I^*(K_i^d)}$. Set
$G_i=G'_i\cup B$ for $i=1,\ldots,s$. It is easy to see that the set
of facets of $\Delta_{I^*(\mathcal{H})}$ is $\{F_i\}_{i=1}^t\cup
\{G_j\}_{j=1}^s$. We claim that the ordering
$G_1<\cdots<G_s<F_1<\cdots<F_t$ is a shelling of
$\Delta_{I^*(\mathcal{H})}$. Let $G_i<F_j$, where
$G_i=\mathcal{X}(\mathcal{H})\setminus E_1$ and
$F_j=\mathcal{X}(\mathcal{H})\setminus E_2$ for some edges $E_1$ of
$K_i^d$ and $E_2$ of $K_m^d$. Let $E_1=\{w_1,\ldots,w_d\}$ and
$E_2=\{v,v_1,\ldots,v_{d-1}\}$, where $v_1,\ldots,v_{d-1}\in
\mathcal{X}(K_j^d)$. Then there exists $1\leq l\leq d$ such that
$w_l\in F_j\setminus G_i$. Set $E_3=\{w_l,v_1,\ldots,v_{d-1}\}$,
then $\mathcal{X}(\mathcal{H})\setminus E_3=G_k$ for some $k$ and
$F_j\setminus G_k=\{w_l\}$.

Now, assume $(i)$ and $(ii)$ does not hold. Then $i\geq d$, $m-j\geq 2$
and $i-j\geq 2$. We first claim that if $j\leq d-2$, there is no
shelling: Consider the intersection $G_j\cap F_i$ (same notation as above). 
The two facets here correspond to two edges in $\mc{H}$, one from $K_m^d$ and 
one from $K_i^d$. These two edges can at most have $j$ elements in common. 
Hence, by considering the set complements of these edges we realize that the 
two facets can at most have $|\mc{X(H)}|-d-2$ vertices in common. This shows 
that no ordering of the $F_i$'s and the $G_j$'s can be a shelling, since 
$\dim(F_i)=\dim(G_j)=\mc{X(H)}-d-1$ for every $i=1,\ldots,t$, $j=1,\ldots,s$.

So, we assume $j\geq d-2$. Let $v_1,v_2\in \mathcal{X}(K_m^d)\setminus
\mathcal{X}(K_j^d)$, $w_1,w_2\in \mathcal{X}(K_i^d)\setminus\mathcal{X}(K_j^d)$ and
$u_1,\ldots,u_{d-2}\in \mathcal{X}(K_j^d)$. To finish the proof by contradiction, we 
assume $\Delta_{I^*(\mathcal{H})}$ is shellable. Consider two edges 
$E_1=\{u_1,\ldots,u_{d-2},v_1,v_2\}$ and $E_2=\{u_1,\ldots,u_{d-2},w_1,w_2\}$. Then
$F_i=\mathcal{X}(\mathcal{H})\setminus E_i$ for $i=1,2$ are facets of $\Delta_{I^*(\mathcal{H})}$. 
Without loss of generality may assume $F_1<F_2$. Hence there exists vertex $v\in F_2\setminus F_1$ 
and facet $F_3$ such that $F_2\setminus F_3=\{v\}$. Let
$F_3=\mathcal{X}(\mathcal{H})\setminus E_3$ for some edge $E_3$.
Since $F_2\setminus F_1=\{v_1,v_2\}$, we have $v=v_1$ or $v=v_2$. Therefore 
$E_3\subseteq \mathcal{X}(K_m^d)$. Also $E_3\setminus\{v\}\subseteq E_2$. Thus $w_1\in E_3$ or 
$w_2\in E_3$, which is a contradiction.
\end{proof}

We end this section with a result on the diameter of the complement of a chordal graph. Recall that 
the diameter of a connected graph $\mc{G}$ is defined as 
\[
\mr{diam}(\mc{G})=\max\{\mr{dist}(u,v)\,;\,u,v\in\mc{X(G})\},
\]
where $\mr{dist}(u,v)$ is the number of edges in a shortest path between $u$ to $v$. If $\mc{G}$ is not 
connected we set the diameter to be $\infty$.

\begin{prop}
Let $\mc{G}$ be a connected chordal graph. Then the diameter of the complementary graph $\mc{G}^c$ is 
at most 3.
\end{prop}

\begin{proof}
If $\mr{diam}(\mc{G}^c)>3$ we find vertices $u$ and $v$ with $\mr{dist}(u,v)=4$. Then the induced graph of $\mc{G}^c$ 
on $\{u,v_1,v_2,v_3,v\}$ is the path $uv_1,v_1v_2,v_2v_3,v_3v$. The graph complement of $\mc{G}^c_{\{u,v_1,v_2,v_3,v\}}$ 
contains a 4-cycle without any chord. This contradiction gives our result. 
\end{proof}

\subsection{$d$-uniform hypergraphs and quasi forests}
It is known that a certain class of simplicial complexes, called quasi-trees (see below), and 
chordal graphs, in some sense contain the same information. 
\begin{rem}
Recall that a flag complex is a simplicial complex in which every
minimal non face consists of precisely 2 elements. As one easily
sees, such complex is determined by its 1-skeleton. 
\end{rem}

The following is the content of Lemma 3.1 in \cite{He3}.

\begin{lem}
Let $\Delta$ be a simplicial complex. Then $\Delta$ is a quasi-forest precisely when 
$\Delta=\Delta(\mc{G})$ for some chordal graph $\mc{G}$. In particular, a quasi-forest 
is a flag complex.
\end{lem}

In this section we will see that there is also a close connection between quasi-trees and the class of chordal 
hypergraphs.\\

\begin{defn}[Faridi, \cite{Fa1}, Zheng, \cite{Zh}]
Let $\Delta$ be a simplicial complex. A {\bf subcollection} $\Gamma$, of $\Delta$, is a subcomplex of $\Delta$ such that 
$\mc{F}(\Gamma)\sse\mc{F}(\Delta)$. A facet $F$ of $\Delta$ is called a {\bf leaf} if either $F$ is the only facet of 
$\Delta$, or there exists a facet $G$ in $\Delta$, $G\neq F$, such that $F\cap H\sse F\cap G$ for any facet $H$ in 
$\Delta$, $H\neq F$.\\
Assume $\Delta$ is connected. Then $\Delta$ is called a {\bf tree} if every subcollection of $\Delta$
has a leaf, and $\Delta$ is called a {\bf quasi-tree} if there exists an order 
$F_1,\ldots,F_t$ of the facets of $\Delta$ such that for each $i=1,\ldots,t$, $F_i$ is a leaf of the simplicial complex 
$\langle F_1,\ldots,F_i\rangle$, whose facets are $F_1,\ldots,F_i$. The order $F_1,\ldots, F_t$ is called a 
{\bf leaf order}. A simplicial complex with the property that every connected component is a (quasi-)tree is called a 
(quasi-){\bf forest}.
\end{defn}

\begin{rem}
A tree is a quasi-tree, but the converse need not hold.
\end{rem}

Let $\Delta$ be a simplicial complex. Denote by $\mc{R}_d(\Delta)$ the simplicial complex obtained from 
$\Delta$ by removing every facet $F$ with $1\le \dim F\le d-2$, and all faces $G\sse F$, with 
$1\le\dim G\le\dim F$, that are not faces of some facet of dimension greater than $d-2$. Conversely, denote by 
$\mc{A}_d(\Delta)$ the simplicial complex obtained from $\Delta$ by adding, as a facet, every face of dimension 
$d-2$ that is not already in the complex. 

\begin{lem}
Let $\Delta(\mc{H})$ and $\Delta(\mc{G})$ be the clique complexes of a $d$-uniform hypergraph $\mc{H}$, and a 
graph $\mc{G}$, respectively. Then the following holds:
\begin{itemize}
\item{$\mc{A}_d(\mc{R}_d(\Delta(\mc{H})))=\Delta(\mc{H})$}
\item{$\mc{R}_{d'}(\mc{A}_{d'}(\Delta(\mc{G})))=\Delta(\mc{G})$, where 
$d'-1\le\min\{\dim F; F\,\mr{facet\,in}\,\Delta(\mc{G})$\}.}
\end{itemize}
\end{lem}

\begin{proof}
This follows immediately from the definition of $\Delta(\mc{H})$.
\end{proof}

\begin{lem}\label{isolated}
Let $\mc{H}=\mc{G}\cup_{K_j^d}K_i^d$ be a chordal hypergraph. If $i<d$ (that is $K_i^d$ is consists of $i$ 
isoloated vertices), we may exchange the attaching of $K_i^d$ to $K_j^d$, with $i-j$ attachings of the form 
\[
\mc{H}'=\mc{G}'\cup_{K_0^d}K_1^d
\]
\end{lem}

\begin{proof}
This is clear, since either way, we are just adding a number of isolated vertices.
\end{proof}

\begin{prop}
Let $\mc{H}$ be chordal hypergraph, and let $\mc{G}$ be a chordal graph. Then the following holds:
\begin{itemize}
\item[$(i)$]{$\mc{R}_d(\Delta(\mc{H}))$ is the clique complex of a chordal graph.}
\item[$(ii)$]{$\mc{A}_{d'}(\Delta(\mc{G}))$ is, for any $d'$, the clique complex of a $d'$-uniform 
chordal hypergraph.}
\end{itemize}

\end{prop}

\begin{proof}
$(i)$\quad A chordal hypergraph $\mc{H}$ may, according to its inductive construction, be represented by a 
sequence of pairs of $d$-complete hypergraphs
\[
(K_{0}^d,K_{i_1}^d),\ldots,(K_{j_t}^d,K_{i_t}^d), 
\]
where in each step of the construction of $\mc{H}$, $K_{i_s}^d$ is attached to $K_{j_s}^d$. We assume that in the 
construction of $\mc{H}$, Lemma \ref{isolated} has been used if necessary. Then every $d$-complete hypergraph in the sequence $(K_{0}^d,K_{i_1}^d),\ldots,(K_{j_t}^d,K_{i_t}^d)$ yields a complete graph, and, by considering 
the facets, it is clear that $\mc{R}_d(\Delta(\mc{H}))$ is the complex of the chordal graph that is represented by the 
sequence of pairs $(K_{0},K_{i_1}),\ldots,(K_{j_t},K_{i_t})$. This proves $(i)$.

Now, let $(K_{0},K_{j_1}),\ldots,(K_{j_t},K_{i_t})$ denote a chordal graph $\mc{G}$. If $d'-2\ge\dim\Delta(\mc{G})$, 
then the claim $(ii)$ is trivial, so we assume $d'-2<\dim\Delta(\mc{G})$. It is obvious that 
$\mc{A}_{d'}(\Delta(\mc{G}))$ will be the complex of a $d'$-uniform hypergraph $\mc{H}$, since every minimal nonface 
has dimension $d'-1$. We now show that $\mc{H}$ is chordal. We do this by constructing a sequence of pairs 
$(K_{0}^d,K_{i_1}^d),\ldots,(K_{j_r}^d,K_{i_r}^d)$, $r\ge t$, from the sequence 
$(K_{0},K_{i_1}),\ldots,(K_{j_t},K_{i_t})$, and showing that this sequence actually defines $\mc{H}$.

First note that if $i_s\ge d'$, a complete graph $K_{i_s}$ immediately yields a $d'$-complete hypergraph $K_{i_s}^{d'}$. 
For such $i_s$, we get a pair $(K_{j_s}^{d'},K_{i_s}^{d'})$, corresponding to the pair $(K_{j_s},K_{i_s})$ in the sequence 
representing $\mc{G}$. If $i_s<d'$, we may instead associate to the pair $(K_{j_s},K_{i_s})$ a sequence of ``trivial pairs'', 
as in Lemma \ref{isolated}. Continuing in this way, we obtain a sequence $(K_{j_1}^d,K_{i_i}^d),\ldots,(K_{j_r}^d,K_{i_r}^d)$, 
representing a $d'$-uniform chordal hypergraph $\mc{H}'$. 

The $d'$-uniform chordal hypergraph that correspond to the constructed sequence yields the same complex as $\mc{H}$, and hence 
we conclude that they must be the same.  
\end{proof}

\begin{cor}
To every chordal hypergraph $\mc{H}$ we may associate a quasi-forest $\Delta$, and vice versa.
\end{cor}

\begin{proof}
If $\Delta$ is a quasi-forest, then $\Delta=\Delta(\mc{G})$ for some chordal graph 
(\cite{He3}, Lemma 3.1). Then, according to the proposition, we may associate to $\Delta$ the chordal 
hypergraph $\mc{H}$ whose clique complex is the complex $\mc{A}_{d'}(\Delta(\mc{G}))$ in the proposition. 
Conversely, given a chordal hypergraph $\mc{H}$ we may associate to it the quasi-forest $\mc{R}_d(\Delta(\mc{H}))$ 
from the proposition.
\end{proof}

\section{Homologically connected hypergraphs, connectivity, and depth}\label{connectivity}

For graphs and simplicial complexes there is a natural notion of being connected. 
This property may be described purely in terms of 0-homologies of certain chain 
complexes. Furthermore, the notion of being connected is very well behaved in the 
sense that if we choose the coefficients in the associated chain complex from a 
field $k$, it does not depend on the characteristic of $k$. This is one reason that 
arguments involving connectedness sometimes are very useful if one is trying to prove 
something about a graph or a simplicial complex. In \cite{Ja}, S.~Jacques deduces 
some lower bounds on Betti numbers of graph algebras. The arguments used there are 
based on the connectedness property of graphs. In this section we define in a 
homological fashion a concept of connected hypergraph.\\

\begin{defn}
Let $\mc{H}$ be a $d$-uniform hypergraph and $k$ be a field. The {\bf connectivity of $\mc{H}$ over $k$}, 
$\mr{con}(\mc{H})$, is defined as
\[
\mr{con}(\mc{H})=\min\{|V|; V\sse [n], \dim\tilde{H}_{d-2}((\Delta(\mc{H}))_{[n]\ssm V};k)\neq0\}.
\]
\end{defn}

\begin{defn}
Let $k$ be a field. If $\mc{H}$ is a $d$-uniform hypergraph with non zero connectivity over $k$, we say that 
$\mc{H}$ is {\bf homologically connected over $k$}. If $\mc{H}$ is homologically connected over every field, 
we say that $\mc{H}$ is homologically connected.
\end{defn}

Note that in the case of graphs, this is the usual notion of connectedness. Also, in terms of homological 
connectedness, the connectivity of a $d$-uniform hypergraph $\mc{H}$, is the cardinality of a minimal 
disconnecting set of vertices.

\begin{prop}
If $\mc{H}$ is homologically connected over $\mathbb{Q}$, it is homologically connected over every field $k$.
\end{prop}

\begin{proof}
By the Universal Coefficient Theorem we have
\[
\tilde{H}_{i}(\Delta(\mc{H});k)\cong\tilde{H}_{i}(\Delta(\mc{H});\mathbb{Q})\otimes k\oplus
\mr{Tor}_1^{\mathbb{Z}}(\tilde{H}_{i-1}(\Delta(\mc{H})),k).
\]
One should note that when we consider a complex $\Delta_{\mc{H}}$ of a non empty $d$-uniform hypergraph, 
$\tilde{H}_{l}(\Delta(\mc{H});k)=0$ for every $l\le d-3$ over every field $k$.
\end{proof}

Recall {\it Hochster's formula}.

\begin{thm}[Hochster's formula] Let $R/I_\Delta$ be the Stanley--Reisner ring of a
simplicial complex $\Delta$. The non zero Betti numbers, 
$\beta_{i,{\bf j}}(R/I_\Delta)=\dim\mr{Tor}_i^R(R/I_{\Delta},k)_{{\bf j}}$, of $R/I_\Delta$, are only in 
squarefree degrees $\bf j$ and may be expressed as
\[
\beta_{i,{\bf j}}(R/I_\Delta)=\dim_k\tilde{H}_{|{\bf
j}|-i-1}(\Delta_{\bf j};k).
\]
Hence the total $i$'th Betti number may be expressed as
\[
\beta_i(R/I_\Delta)=\sum_{V\sse[n]}\dim\tilde{H}_{|V|-i-1}(\Delta_V;k).
\]
\end{thm}
\begin{proof}
See \cite{BH}, Theorem 5.5.1.
\end{proof}
From this it follows that
\[
\beta_{i,j}(R/I_\Delta)=\sum_{\substack{V\sse [n]\\
|V|=j}}\dim\tilde{H}_{|V|-i-1}(\Delta_V;k).
\]

\begin{prop}
If $\mc{G}$ is an induced hypergraph of a $d$-uniform hypergraph $\mc{H}$, such that $\mc{G}$ is not 
homologically connected over $k$, then 
\[
\beta_{|\mc{X(G)}|-d+1}(\mc{H})\neq 0
\]
\end{prop}

\begin{proof}
We will use the fact that $(\Delta(\mc{H}))_V=\Delta(\mc{H}_V)$. Consider Hochster's 
formula with $i=|\mc{X(G)}|-d+1$;
\[
\beta_{|\mc{X(G)}|-d+1}(R/I_{\Delta(\mc{H})})=\sum_{V\sse\mc{X(H)}}\dim_k\tilde{H}_{|V|-|\mc{X(G)}|+d-2}
(\Delta(\mc{H}_V);k)\ge
\]
\[\dim_k\tilde{H}_{d-2}(\Delta_{\mc{X(G)}};k)>0.
\]
\end{proof}

Recall the {\it Auslander-Buchsbaum} formula: If $M$ is a finitely generated $R$-module with $\mr{pd}_R(M)<\infty$, 
then $\mr{pd}_R(M)+\mr{depth}_R(M)=\mr{depth}_R(R)$. For a proof, see \cite{BH}, Theorem 1.3.3.

\begin{cor}
If $\mc{G}$ is an induced hypergraph of a $d$-uniform hypergraph $\mc{H}$, such that $\mc{G}$ is not 
homologically connected over $k$, then 
\[
|\mc{X(G)}|-d+1\le\mr{pd}_R(R/I_{\Delta(\mc{H})})\le n ,
\]
\[
0\le\mr{depth}_R(R/I_{\Delta(\mc{H})})\le n-|\mc{X(G)}|+d-1,
\]

where $n=|\mc{X(H)}|$.
\end{cor}

\begin{proof}
It is well know that (Hilbert's syzygy theorem) $n\ge\mr{pd}_R(R/I_{\Delta(\mc{H})})$. Furthermore, 
according to the lemma, $\beta_{|\mc{X(G)}|-d+1}(\mc{H})>0$. This gives the first assertion. The second 
follows from the first using the Auslander-Buchsbaum formula.
\end{proof}

\begin{cor}
If $\mc{H}$ is a $d$-uniform hypergraph that is not homologically connected over $k$, then
\[
n-d+1\le\mr{pd}_R(R/I_{\Delta(\mc{H})})\le n ,
\]
\[
0\le\mr{depth}_R(R/I_{\Delta(\mc{H})})\le d-1,
\]
where $n=|\mc{X(H)}|$.
\end{cor}

If $\mc{H}$ is a $d$-uniform hypergraph that is not homologically connected, we will see in Theorem \ref{homconn} 
below, the two inequalities in the above corollary may in fact be exchanged with two equalities. First, we prove the 
following theorem, which connects the depth of the Stanley--Reisner ring $R/I_{\Delta(\mc{H})}$, with the 
connectivity of $\mc{H}$.

\begin{thm}\label{conn/depth}
Let $\Delta(\mc{H})$ be the complex of a $d$-uniform hypergraph $\mc{H}$ and put 
$g=\mr{depth}_R(R/I_{\Delta(\mc{H})})$. Then,
\[
\mr{con}(\mc{H})=g-d+r+1,
\]
where $r$ is the minimal number such that $\beta_{n-g-r,n-g-r+d-1}(R/I_{\Delta(\mc{H})},k)\neq0$. That is, 
$r$ is the minimal number such that there exists a $V\sse [n]$, $|V|=n-(g-d+r+1)$ with 
$\tilde{H}_{d-2}(\Delta_V;k)\neq 0$ 
\end{thm}

\begin{rem}
If $\mc{H}$ is a $d$-uniform hypergraph, recall that the {\it linear strand} of a resolution of 
$R/I_{\Delta(\mc{H})}$ (or for short, the linear strand of $R/I_{\Delta(\mc{H})}$) is the part of the 
resolution that is of degrees $(i,i+d-1)$. Note that 
$r=\mr{pd}_R(R/I_{\Delta(\mc{H})})-\max\{i\,;\,\beta_{i,i+d-1}(R/I_{\Delta(\mc{H})})\neq0\}$.
\end{rem}

\begin{proof}
We know that $\mr{Tor}_{n-g}^R(R/I_{\Delta(\mc{H})},k)\neq0$, but $\mr{Tor}_{n-i}^R(R/I_{\Delta(\mc{H})};k)=0$ 
for every $i<g$. In particular, $\mr{Tor}_{n-i}^R(R/I_{\Delta(\mc{H})};k)_j=0$ in every degree $j$ if $i<g$. 
This gives, via Hochster's formula, that $\tilde{H}_{|V|-(n-i+1)}(\Delta_V;k)=0$ for every $V\sse [n]$, $i<g$, 
and that there exists a $V\sse [n]$ such that $\tilde{H}_{|V|-(n-g+1)}(\Delta_V;k)\neq0$. Let $r\ge0$ be the 
minimal number such that $\mr{Tor}_{n-(g+r)}^R(R/I_{\Delta(\mc{H})};k)_j\neq0$ for $j=n-(g-d+r+1)$. This is 
the same thing as saying that there exists a $V\sse [n]$, $|V|=n-(g-d+r+1)$, such that 
$\tilde{H}_{d-2}(\Delta_V;k)\neq0$ but at the same time, for any $V\sse [n]$, $|V|>n-(g-d+r+1)$, the homology 
of $\Delta_V$ in degree $d-2$ is zero. This means precisely that $\mr{con}(\mc{H})=g-d+r+1$. 
\end{proof}

If $\mc{H}$ is 2-uniform (that is, if $\mc{H}$ is an ordinary simple graph) and we have linear resolution, the 
following is Lemma 3 in \cite{F1}.

\begin{cor}\label{baba}
Let $\mc{H}$ be a $d$-uniform hypergraph and suppose the length of the linear strand of $R/I_{\Delta(\mc{H})}$ is 
maximal. Then 
\[
\mr{depth}_R(R/I_{\Delta(\mc{H})})=\mr{con}(\mc{H})+d-1.
\]
\end{cor}

\begin{thm}\label{homconn}
Let $\mc{H}$ be a $d$-uniform hypergraph. Then $\mc{H}$ is not homologically connected over $k$ 
precisely when
\[
\mr{pd}_R(R/I_{\Delta(\mc{H})})=n-d+1,
\]
\[
\mr{depth}_R(R/I_{\Delta(\mc{H})})=d-1,
\]
where $n=|\mc{X(H)}|$, and the length of the linear strand of $R/I_{\Delta(\mc{H})}$ is maximal.
\end{thm}

\begin{proof}
We know that $n-d+1\le\mr{pd}_R(R/I_{\Delta(\mc{H})})\le n $. Put $\mr{pd}_R(R/I_{\Delta(\mc{H})})=n-r$, 
$0\le r\le d-1$. Hochster's formula gives
\[
\beta_{n-r}(R/I_{\Delta(\mc{H})})=\sum_{V\sse\mc{X(H)}}\dim_k\tilde{H}_{|V|-(n-r)-1}
(\Delta(\mc{H}_V);k).
\]
If $r\le d-2$ we get that $|V|-(n-r)-1\le |V|-n+d-3\le d-3$. But 
$\tilde{H}_{l}(\Delta(\mc{H}_V);k)=0$ for all $l\le d-3$ and for all $V\sse\mc{X(H)}$.\\

The last claim follows from Corollary \ref{baba}, since being not homologically connected, is the same 
thing as having connectivity 0.
\end{proof}

\begin{exam}
Since homologically connected and connected are the same things for an ordinary 
simple graph $\mc{G}$, we have $\mr{pd}_R(R/I_{\Delta(\mc{G}}))=n-1$ and $\mr{depth}_R(R/I_{\Delta(\mc{G})})=1$ 
for any simple graph $\mc{G}$ that is not connected. Furthermore, the length of the linear strand of 
$R/I_{\Delta_{\mc{G}}}$ is maximal. This special case of Theorem \ref{homconn} is Theorem 4.2.6 in \cite{Ja}.
\end{exam}

\begin{cor}\label{CM}
Let $\mc{H}$ be a $d$-uniform hypergraph. If $R/I_{\Delta(\mc{H})}$ is Cohen--Macaulay of dimension at 
least $d$, then $\mc{H}$ has non zero connectivity. Put another way, the only $d$-uniform hypergraph 
$\mc{H}$ with connectivity 0 such that $R/I_{\Delta(\mc{H})}$ is Cohen--Macaulay, is the discrete hypergraph.
\end{cor}

\begin{exam}
In \cite{E1} we considered several kinds of complete hypergraphs. For example the 
$d$-complete multipartite hypergraph $K_{n_1,\ldots,n_t}^d$. This is the hypergraph on vertex set 
$\mc{X}=[n_1]\sqcup\cdots\sqcup [n_t]$, where $\sqcup$ denotes disjoint union, and edge set consisting 
of every $d$-set of $\mc{X}$ (that is, every subset of $\mc{X}$ of cardinality $d$) that does not lie 
entirely inside one of the $[n_i]$'s. This is a natural generalization of the usual complete 
mulitpartite graph $K_{n_1,\ldots,n_t}$. Precisely as $(K_{n_1,\ldots,n_t})^c$, the complement 
$(K_{n_1,\ldots,n_t}^d)^c$ of the $d$-complete multipartite hypergraph is not homologically connected.\\
In Propositions 3.9 and 3.20 in \cite{E1}, we determined when a couple of 
such complete hypergraphs in addition to having linear resolutions also has the Cohen--Macaulay property. The 
conclusion there is that the only case in which this happens, is in the extremal case when the considered 
hypergraph in fact is isomorphic to a $d$-complete hypergraph. This fact now follows immediately from the above 
corollary, since it is easily seen (by computing the Betti numbers) that the considered hypergraphs are not 
homologically connected.
\end{exam}

\begin{prop}
Let $\mc{H}$ be a $d$-uniform hypergraph on vertex set $[n]$. Then the Betti number 
$\beta_{n-d+1}(R/I_{\Delta(\mc{H})})$ can be non zero only in degree $n$. Furthermore, it determines whether 
$\mc{H}$ has non zero connectivity or not. 
\end{prop}

\begin{proof}
This follows from Hochster's formula and the fact that $\beta_{i,j}(R/I_{\Delta(\mc{H})})=0$ if $j<i+d-1$. 
\end{proof}

\begin{rem}
In case of ordinary simple graphs, by the above proposition, the number 
$\beta_{n-1}(R/I_{\Delta(\mc{G})})+1$ is the number of connected components of $\mc{G}$.
\end{rem}

\begin{exam}
Let $\mc{H}$ be the 3-uniform hypergraph on vertex set $\{a,b,c,d\}$ and with edge set 
$\mc{E(H)}=\{abc,bcd\}$ (we let $xyz$ denote the edge $\{x,y,z\}$). We may visualize $\mc{H}$ as follows:
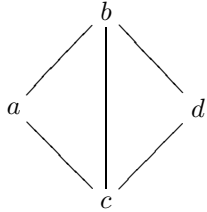
\[
\xymatrix{&b\ar@{-}[dl]\ar@{-}[dd]\ar@{-}[dr]&\\
a\ar@{-}[dr]&&d\ar@{-}[dl]\\
&c&}
\]
By computing the Betti numbers of $R/I_{\Delta(\mc{H})}$ using some suitable computer program, one sees 
that $\beta_{2}(R/I_{\Delta(\mc{H})})=1$. If we add to the edge set the edge $abd$, the resulting hypergraph 
has non zero connectivity.
\end{exam}

\begin{exam}
Let $\mc{H}$ be the 3-uniform hypergraph on vertex set $\{a,b,c,d,e\}$ and with edge set 
$\mc{E(H)}=\{abc,cde\}$. $\mc{H}$ is illustrated below:
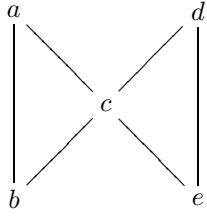
\[
\xymatrix{
a\ar@{-}[dd]\ar@{-}[dr]&&d\ar@{-}[dd]\ar@{-}[dl]\\
&c&\\
b\ar@{-}[ur]&&e\ar@{-}[ul]}
\]
The Betti number $\beta_3(R/I_{\Delta(\mc{H})})=4$ shows that $\mc{H}$ has 0 connectivity. If we add to the edge 
set the edge $acd$, the resulting hypergraph has 4 homologically connected components. If we continue and add the 
edge $bce$, the resulting hypergraph still has 0 connectivity. Adding the edge $abd$ does not create a hypergraph 
with non zero connectivity, but, finally, by adding the edge $abe$ we arrive at a hypergraph with edge set 
$\{abc,cde,acd,bce,abd,abe\}$, that has non zero connectivity.
\end{exam}

\begin{exam}
Let $\mc{H}=(K_n^d)^c$. Then we know that (Theorem 3.1, \cite{E1}) 
$\beta_{i,j}(R/I_{\Delta(\mc{H})})\\=\gf{n},{j}\gf{j-1},{d-1}$. Hence $\mc{H}$ certainly does not have 
non zero connectivity. This is quite natural since it generalizes the fact that the discrete graph on $n$ vertices 
has $n=\gf{n-1},{2-1}+1$ (homologically) connected components.
\end{exam}

\begin{exam}
One of the complete hypergraphs considered in \cite{E1} is the $d(a_1,\ldots,a_t)$-complete 
hypergraph $K_{n_1,\ldots,n_t}^{d(a_1,\ldots,a_t)}$. This hypergraph has vertex set the disjoint union 
$\mc{X}=[n_1]\sqcup\cdots\sqcup [n_t]$ and edge set consisting of all $d$-sets of $\mc{X}$ such that precisely 
$a_i$ elements comes from $[n_i]$. In \cite{E1} it is shown that $R/I(K_{n_1,\ldots,n_t}^{d(a_1,\ldots,a_t)})$ has linear 
resolution and projective dimension $n-d+1$. Hence $(K_{n_1,\ldots,n_t}^{d(a_1,\ldots,a_t)})^c$ has 
connectivity 0.
\end{exam}

\begin{rem}
The above examples shows that the connectivity of an arbitrary $d$-uniform hypergraph, often is 0. This together 
with Corollary \ref{CM} show that $d$-uniform hypergraphs $\mc{H}$ such that $R/I_{\Delta(\mc{H})}$ 
is Cohen--Macualay, indeed are very special.
\end{rem}

In Example 5.17 and Example 5.18, we have $\dim\Delta(\mc{H})=2$. It is easy to see that if $\mc{H}$ is a 
$d$-uniform hypergraph such that $\dim\Delta(\mc{H})=d-1\ge1$, then the converse of Corollary \ref{CM} holds. 
That is, if $\dim\Delta(\mc{H})=d-1\ge1$, then $R/I_{\Delta(\mc{H}}$ is Cohen--Macaulay over $k$ precisely when 
$\mc{H}$ has non zero connectivity over $k$. This will follow from the following result (Lemma 7 in \cite{F2}) 
of Fr\"oberg.

\begin{lem}\label{Froberg}
Let $R/I$ be a Stanley--Reisner ring with $\dim R/I=e$ and embedding dimension $n$. Then $R/I$ is Cohen--Macaulay 
if and only if $\tilde{H}_{i}(\Delta_V;k)=0$ for every $i$ and $V\sse [n]$ such that $|V|=n-e+i+2$.
\end{lem}

\begin{rem}
An easy way to think of this lemma is as follows: First note that $|V|=n-e+i+2$ gives $i=0,\ldots,e-2$. The 
claim in the lemma is now ``symmetric'' relative to this sequence of indexes. $\tilde{H}_{0}(\Delta_V;k)$ should 
be zero for $|V|=n-(e-2)$, $\tilde{H}_{1}(\Delta_V;k)$ should be zero whenever $|V|=n-(e-2)-1$, 
$\tilde{H}_{2}(\Delta_V;k)$ should be zero whenever $|V|=n-(e-2)-2$, a.s.o.
\end{rem}

If $\Delta=\Delta(\mc{H})$, the lemma of Fr\"oberg gives us the following, which is completely analogous to that 
considered right after Lemma 7 in \cite{F2}. If $\dim\Delta(\mc{H})=d-2$, the complex is always Cohen--Macaulay. 
This follows since the claim in the lemma in this case is that reduced homology in degree -1 is zero (we consider 
non empty complexes). The claim could also be easily verified by noting that $\Delta(\mc{H})$ in this case is the 
independence complex $\Delta_{K_n^d}$ of some $d$-complete hypergraph, see \cite{E1}, Corollary 3.2. 
Assume $\dim\Delta(\mc{H})=d-1$. In this case the condition in the lemma is that $\mc{H}$ be homologically connected 
(i.e. that $\tilde{H}_{d-2}(\Delta(\mc{H});k)=0$). If $\dim\Delta(\mc{H})=d$ the condition in the lemma says that $\tilde{H}_{d-1}(\Delta(\mc{H});k)=0$ and that $\tilde{H}_{d-2}(\Delta(\mc{H}_V);k)=0$ for every $V\sse [n]$ with 
$|V|=n-1$.\\

Consider Lemma \ref{Froberg} for a complex $\Delta(\mc{H})$ with linear resolution. Since induced complexes 
$(\Delta(\mc{H}))_V=\Delta(\mc{H}_V)$ can only have homology in degree $d-2$, one gets:

\begin{cor}
Let $\mc{H}$ be a $d$-uniform hypergraphs such that $\dim(R/I_{\Delta(\mc{H})})=e$ and 
$R/I_{\Delta(\mc{H})}$ has linear resolution. Then it is also Cohen--Macaulay if 
and only if $\tilde{H}_{d-2}(\Delta(\mc{H}_V);k)=0$ for every $V\sse [n]$ with $|V|=n-(e-d)$. Furthermore, 
in this case we have that
\[
e=\mr{con}(\mc{H})+d-1.
\]
\end{cor}

\section{$d$-shellability}\label{d-shellability}

Pure shellable simplicial complexes is somewhat of a cornerstone of combinatorial commutative algebra. This is 
perhaps mostly since in some situations they provide a nice non-technical (not always an easy though) way of 
showing that a complex $\Delta$ is Cohen--Macaulay. Also, the concept has many times been succesfully used to prove, 
via Alexander duality, that certain rings have linear resolutions, indeed, even linear quotients.\\

We start by recalling the definition of shellability, pure and non-pure. We use the following notation: Given 
a finite collection $\{F_1,\ldots,F_t\}$ of non empty subsets of $[n]$, we denote by 
$\langle F_1,\ldots,F_t\rangle$ the simplicial complex with $\mc{F}(\Delta)=\{F_1,\ldots,F_t\}$.

\begin{defn}
Let $\Delta$ be a simplicial complex on $[n]$ with $\mc{F}(\Delta)=\{F_1,\ldots,F_t\}$. $\Delta$ is called 
pure shellable if 
\begin{itemize}
\item[$(i)$]{
$|F_i|=|F_j|$ for every pair of indices $1\le i<j\le t$.}
\item[$(ii)$]{
There exists an ordering $F_1,\ldots,F_t$ of the facets such that 
$\langle F_j\rangle\cap\langle F_1,\ldots F_{j-1}\rangle$ is generated by a non-empty set of proper maximal faces 
of $\langle F_j\rangle$ for every $j=2,\ldots,t$.
}
\end{itemize}
\end{defn}

A simplicial complex $\Delta$ is by definition called non-pure shellable if $(ii)$ but not necessarily $(i)$, 
holds in the above definition. 

Henceforth unless otherwise is stated, by shellable we mean shellable in the non-pure sense.

\begin{rem}
It is well known (see for example \cite{BH}, Theorem 5.1.13) that pure shellability implies Cohen--Macaulayness. This 
follows from Corollary \ref{lin.quot} below and the {\it Eagon-Reiner Theorem} (\cite{ER}, Theorem 3).
\end{rem}

In the following two definitions we introduce the concepts of $d$-shellability and 
$d$-quotients. .

\begin{defn}\label{d-shellable}
Let $\Delta$ be a simplicial complex on $[n]$. $\Delta$ is called $d$-shellable if its facets can be ordered 
$F_1,\ldots,F_t$, such that $\langle F_j\rangle\cap\langle F_1,\ldots,F_{j-1}\rangle$ is generated by a non-empty 
set of proper faces of $\langle F_j\rangle$ of dimension $|F_j|-d-1$ for every $j=2,\ldots,t$.
\end{defn}

\begin{rem}
The concepts of being 1-shellable and shellable coincides. If $\Delta$ is a simplicial complex, a linear ordering of 
$\mc{F}(\Delta)$ satisfying the conditions of Definition \ref{d-shellable} is called a {\bf $d$-shelling} of $\Delta$.
\end{rem}

\begin{defn}
Let $I$ be a monomial ideal. We say that $I$ has $d$-quotients if there exists an ordering 
$x^{m_1}\le\cdots\le x^{m_t}$ of the minimal generators of $I$, such that if we for 
$s=1,\ldots,t$, put $I_s=(x^{m_1},\ldots,x^{m_s})$, then for every $s$ there are monomials 
$x^{b_{s_i}}$, $i=1,\ldots,r_s$, $\deg x^{b_{s_i}}=d$ for all $i$, such that
\[
I_{s-1}:x^{m_s}=(x^{b_{s_1}},\ldots,x^{b_{s_{r_s}}}).
\]
\end{defn}

The motivation behind these definitions is the following well known theorem, which we 
generalize below.

\begin{thm}
Let $I=(x^{m_1},\ldots,x^{m_t})$ be a squarefree monomial ideal. Then $I$ has linear quotients 
(that is, 1-quotients) precisely when the Alexander dual ideal, $I^\ast$, is shellable. In particular, 
if $I^\ast$ is shellable and $\deg x^{m_i}=\deg x^{m_j}$ for every pair $i,j$ of indices, then $I^\ast$ 
is Cohen--Macaulay.
\end{thm}

\begin{exam}
The clique complexes of line hypergraphs $L_{n}^{d,\alpha}$ and of
hypercycles $C_n^{d,\alpha}$ are both $(d-\alpha)$-shellable. 
\end{exam}

\begin{thm}\label{d-quot/d-shell}
Let $I$ be a squarefree monomial ideal. Then $I$ has $d$-quotients precisely when the Alexander dual ideal, 
$I^\ast$, is $d$-shellable.
\end{thm}

\begin{proof}
Let $I=(x^{m_1},\ldots,x^{m_t})$, where the $x^{m_i}$ are the minimal generators. By definition, it is clear that 
the set of facets of $\Delta^\ast$, the Stanley--Reisner complex of $I^\ast$, is $\mc{F}(\Delta^\ast)=\{F_1,\ldots,F_t\}$, 
where $F_i=[n]\ssm m_i$ for $i=1,\ldots,t$. With the notation clear, the proof is almost tautological.

Assume $I$ has $d$-quotients.  If for every $1\le i<j\le t$, $x^{a_{j_i}}$ denotes the minimal generator of $(x^{m_i}):x^{m_j}$, 
then (possibly after re-indexing) $I_{j-1}:x^{m_j}$ is minimally generated by the set $x^{a_{j_1}},\ldots,x^{a_{j_r}}$, 
for some $r\le j-1$. This is equivalent to saying that the sets $a_{j_\alpha}$, $\alpha=1,\ldots,j_r$, that all have cardinality 
$d$ by assumption, are precisely the minimal subsets of $[n]$ such that $F_j\ssm a_{j_\alpha}\sse F_i$ for some $1\le i<j$, 
and that $\langle F_j\rangle\cap\langle F_1,\ldots,F_{j-1}\rangle$ is pure of dimension $|F_j|-d-1$ and equals 
$\langle F_j\ssm a_{j_1},\ldots,F_j\ssm a_{j_r}\rangle$.

The converse is proved by a similar argument: Assume $\Delta^\ast$ is $d$-shellable, and let 
$\mc{F}(\Delta^\ast)=\{F_1,\ldots,F_t\}$. Put $m_i=[n]\ssm F_i$. Then the Alexander dual ideal $I$ of $I^\ast$, is minimally 
generated by the monomials $x^{m_i}$, $i=1,\ldots,t$. For every $j=2,\ldots,t$, we let $a_{j_\alpha}$, $\alpha=1,\ldots,j_r$ 
denote the subsets of $F_j$ that one has to remove in order for $F_j\ssm a_{j_\alpha}$ to be a generator of 
$\langle F_j\rangle\cap\langle F_1,\ldots,F_{j-1}\rangle$. Then the monomials $x^{a_{j_\alpha}}$ are precisely the minimal 
generators of $I_{j-1}:x^{m_j}$. 
\end{proof}

The following theorem occurs frequently in the literature. It shows that simplicial complexes that are 1-shellable 
may be defined in (at least) three equivalent ways:

\begin{thm}
Let $\Delta$ be a simplicial complex on vertex set $[n]$, with $\mc{F}(\Delta)=\{F_1,\ldots,F_t\}$. Then the 
following conditions are equivalent:
\begin{itemize}
\item[$(i)$]
{$\Delta$ is shellable and $F_1,\ldots,F_t$ is a shelling.}
\item[$(ii)$]
{For all $i,j$, $1\le i<j\le t$, there exist a vertex $v$ and a $k$ with $1\le k<j$, such that 
$v\in F_j\ssm F_i$ and $F_j\ssm F_k=\{v\}$.}
\item[$(iii)$]
{The set $\{F\in [n]\,|\,F\in\langle F_1,\ldots,F_j\rangle\,,\,F\not\in\langle F_1,\ldots,F_{j-1}\rangle\}$ has a unique 
minimal element for all $2\le i\le t$.}
\end{itemize}

\end{thm}

Two of these statements, slightly modified, remain equivalent in the case of $d$-shellable complexes also for $d>1$.

\begin{thm}
Let $\Delta$ be a simplicial complex on vertex set $[n]$, with $\mc{F}(\Delta)=\{F_1,\ldots,F_t\}$. Then the 
following conditions are equivalent:
\begin{itemize}
\item[$(i)$]
{$\Delta$ is $d$-shellable and $F_1,\ldots,F_t$ a $d$-shelling.}
\item[$(ii)$]
{For all $i,j$, $1\le i<j\le t$, there exist some set $a_j\sse [n]$, $|a_j|=d$, and a $k$ with $1\le k<j$, such that 
$a_j\sse F_j$, $a_j\cap F_i=\emptyset$ and $F_j\ssm F_k=a_j$.}
\end{itemize}
\end{thm}

\begin{proof}
The implication $(i)\Rightarrow (ii)$ follows by considering the proof of Theorem \ref{d-quot/d-shell}. For the 
converse let $F$ be a face of $\langle F_j\rangle\cap\langle F_1,\ldots,F_{j-1}\rangle$. Then $F$ lies in some 
$\langle F_i\rangle$, $i<j$. Let $a_j$ be a set that fits the description in $(ii)$. Then $F$ is also a face of 
$\langle F_j\ssm a_j\rangle$ so $\langle F_j\rangle\cap\langle F_1,\ldots,F_{j-1}\rangle$ is pure of dimension 
$|F_j|-d-1$. 
\end{proof}

As for shellable complexes, links of faces of $d$-shellable complexes stay $d$-shellable:

\begin{prop}
Let $\Delta$ be a $d$-shellable complex and $F$ a face of $\Delta$. Then $\mr{lk}_\Delta(F)$ is again 
$d$-shellable.
\end{prop}

\begin{proof}
Assume $F_1,\ldots,F_t$ is a $d$-shelling of $\Delta$ and that the face $F$ lies is the facets 
$F_{i_1},\ldots,F_{i_r}$. Put $G_{i_j}=F_{i_j}\ssm F$. Then 
$\mr{lk}_\Delta(F)=\{G_{i_1},\ldots,G_{i_r}\}$. If $j\le r$ and $G$ is a face of 
$G_{i_j}\cap\langle G_{i_1},\ldots,G_{i_{j-1}}\rangle$, then $F\cup G$ is a face of 
$F_{i_j}\cap\langle F_1,\ldots,F_{i_{j-1}}\rangle$. Hence, if $G$ is maximal we see that 
$|G|=|F_{j_i}|-|F|-d$, which is our result.
\end{proof}

We now investige the behaviour of the Betti numbers of ideals with $d$-quotients. The following two results are 
more or less obviuos. We record them just since they show that the notion of $d$-quotients is not empty.

\begin{lem}
Let ${\bf y}=y_1,\ldots,y_r$ be a sequence of monomials in $R=k[x_1,\ldots,x_n]$. Then ${\bf y}$ is an $R$-sequence 
precisely when $\mr{gcd}(y_i,y_j)=1$ for every $i\neq j$. 
\end{lem}

\begin{prop}
For every pair of integers $1\le d\le d'$ there exist a squarefree monomial ideal 
$I=(x^{m_1},\ldots,x^{m_t})\sse k[x_1,\ldots,x_n]$, $n$ sufficiently large, $\deg x^{m_i}=d'$ for every $i=1,\ldots,t$, 
such that: if we put $I_s=(x^{m_1},\ldots,x^{m_s})$, $s=1,\ldots,t$, then every colon ideal $I_s:x^{m_{s+1}}$ is 
generated by an $R$-sequence ${\bf x}_s=x^{b_{s_1}},\ldots,x^{b_{s_{r_s}}}$ of squarefree monomials of degree $d$.
\end{prop}

\begin{proof}
Let $M\sse [n]$ be a set such that $|M|=d'-d$. Choose the generators $x^{m_i}$ such that $|m_i|=d'$ for 
every $i=1,\ldots,t$ and $m_i\cap m_j=M$ for every $i\neq j$. 
\end{proof}

Splittable monomial ideals, introduced by Eliahou and Kervaire in \cite{Eli}, has been studied in for example \cite{VT5,VT,VT2}. 
This class of ideals is well behaved in the sense that their Betti numbers satisfy the {\it Eliahou-Kervaire formula}, see 
\cite{Eli} Proposition 3.1. The following definition (that is Definition 1.1 in \cite{VT5}), captures the content of the 
Eliahou-Kervaire formula in an axiomatic way. 

\begin{defn}
Let $I$, $J$ and $K$ be monomial ideals such that $\mc{G}(I)$ is the disjoint union of $\mc{G}(J)$ and 
$\mc{G}(K)$. Then $I=J+K$ is a {\bf Betti splitting} if
\[
\beta_{i,j}(I)=\beta_{i.j}(J)+\beta_{i,j}(K)+\beta_{i-1,j}(J\cap K)
\]
for all $i\in\mathbb{N}$ and (multi)degrees $j$.
\end{defn}

It is easy to see that a monomial ideal $I$ with linear quotients has a very natural Betti splitting. This is the core of 
the fact that the minimal free resolution of $I$ is a mapping cone. The connection between ``being a mapping cone'' and 
``having a Betti splitting'', is described in \cite{VT5} Proposition 2.1.\\

\begin{thm}\label{about quotients}
Let $I=(x^{m_1},\ldots,x^{m_t})$, $\deg x^{m_i}=d'$ for every $i=1,\ldots,t$, be a squarefree monomial ideal with 
$d$-quotients, $d\le d'$, and put $I_s=(x^{m_1},\ldots,x^{m_s})$, $s=1,\ldots,t$. Then
\begin{itemize}
\item[$(i)$]{$\beta_{i,j}(R/(I_{s-1}:x^{m_{s}})(-d'))$ and $\beta_{i,j}(R/I_{s-1})$ are not non zero in any common 
degree $j$ for any $i\ge 2$, $s=2,\ldots,t$. Hence $I_s=I_{s-1}+(x^{m_s})$ is a Betti splitting.}
\item[$(ii)$]{For all $i$, $2\le i\le\mr{pd}_R(R/I)$, we have
\[
\beta_i(R/I)=\sum_{s=2}^{t}\beta_{i-1}(R/(I_{s-1}:x^{m_{s}})(-d')).
\]}
\end{itemize}
\end{thm}

\begin{proof}
$(ii)$ is a consequence of $(i)$ since if we assume that $(i)$ holds, then for every $s=2,\ldots,t$ we 
have an exact sequence 
\[
0\to R/(I_{s-1}:x^{m_{s}})(-d')\mapr{x^{m_{s}}} R/I_{s-1}\to R/I_{s}\to0,
\]
where the first map is multiplication by $x^{m_{s}}$. It follows from the long exact Tor-sequence that 
$\beta_i(R/I_{t})=\beta_i(R/I_{t-1})+\beta_{i-1}(R/(I_{t-1}:x^{m_{t}})(-d'))$. Noting that $\beta_2(R/I_1)=0$, 
$(ii)$ now follows by induction.\\

To prove $(i)$, let $2\le r\le t$ and consider the following exact secuence
\begin{eqnarray}
0\to I_{r-1}\cap (x^{m_r})\to I_{r-1}\oplus (x^{m_r})\to I_r\to0.
\end{eqnarray}
The non trivial maps are $x\mapsto (x,-x)$ and $(x,y)\mapsto x+y$.

Let $F'.$ and $G'.$ be the minimal free resolutions of $I_{r-1}\cap (x^{m_r})$ and $I_{r-1}\oplus (x^{m_r})$ 
respectively. It follows from Proposition 2.1 in \cite{VT5} that $I_r=I_{r-1}+(x^{m_r})$ is a Betti splitting 
precisely when the mapping cone, cone($\alpha$), of the lifting $\alpha:F'.\to G'.$ of the left map in the above 
exact sequence is the minimal free resolution of $I_r$.
 
Given a monomial ideal $J=(x^{k_1},\ldots,x^{k_u})$ with linear quotients and $\deg(x^{k_1})\le\cdots\le\deg(x^{k_u})$, 
it is known (\cite{He5}) that the minimal free resolution of $R/J$ is the mapping cone of the lifting of the map $R/((x^{k_1},\ldots,x^{k_{u-1}}):x^{k_u})\mapr{x^{k_u}}R/(x^{k_1},\ldots,x^{k_{u-1}})$ to the corresponding minimal 
free resolutions. This is still true if we use $d$-quotients instead, and is easily verified. Now, consider the 
ideals $I_{r-1}\cap (x^{m_r})$ and $I_{r-1}:x^{m_r}$. By looking at the generators of these two ideals, it is clear 
that we have an homogeneous $R$-module isomorphism
\[
(I_{r-1}:x^{m_r})(-d')\cong I_{r-1}\cap(x^{m_r}).
\]

Let $F.$ and $G.$ be the minimal free resolutions of $R/(I_{r-1}:x^{m_r})(-d')$ and $R/I_{r-1}$ respectively, and 
$\alpha:F.\to G.$ a lifting of the map $R/(I_{r-1}:x^{m_r})(-d')\mapr{x^{m_r}} R/I_{r-1}$. Note that the minimal free 
resolutions of $I_{r-1}$ and $I_{r-1}\oplus (x^{m_r})$ only differ in a very simple way at the bottom degrees. Using this 
and the above isomorphism, we realize that the mapping cone of the lifting of the map 
$I_{r-1}\cap(x^{m_r})\to I_{r-1}\oplus (x^{m_r})$, essentially is obtained by truncating the mapping cone of $\alpha$. 
Hence, since we know that cone($\alpha$) is the minimal free resolution of $R/I_r$, this new mapping cone is the minimal 
free resolution of $I_r$ and $I_r=I_{r-1}+(x^{m_r})$ is a Betti splitting.\\
\end{proof}

\begin{exam}
$I=(abc,cde,bef,adf)\sse k[a,b,c,d,e,f]$ is an ideal with 2-quotients. The Betti numbers of $R/I$, 
in homological degrees $1, 2,$ and $3$, are $4,6$ and $3$. The corresponding Betti numbers for $R(-3)/(I_1:cde)$, 
$R(-3)/(I_2:bef)$, and $R(-3)/(I_3:adf)$, are $1, 0$ and $0$; $2, 1$ and $0$; and $3, 2$ and $0$ respectively. 
It is easily verified that these sum up, according to the proposition, to the Betti numbers of $R/I$.
\end{exam}

\begin{cor}
Let $I=(x^{m_1},\ldots,x^{m_t})$, $\deg x^{m_i}=d'$ for every $i=1,\ldots,t$, be a squarefree monomial ideal with 
$d$-quotients, $d\le d'$, and assume the minimal generators of $I_{s-1}:x^{m_{s}}$ forms an $R$-sequence for every 
$s=1,\ldots,t$. Then $\beta_{i,j}(R/I)$ is non zero only for $j=i+d'-1+(i-1)(d-1)$, and for all $i$, 
$2\le i\le\mr{pd}(R/I)$, we have 
\[
\beta_{i,j}(R/I)=\sum_{s=2}^{t}\gf{r_s},{i-1}.
\]
\end{cor}

\begin{proof}
Let $I_{s-1}:x^{m_{s}}=(x^{b_{s_1}},\ldots,x^{b_{s_{r_s}}})$. It is then easy to see that 
$\beta_i(R/(I_{s-1}:x^{m_{s}})(-d'))=\gf{r_s},{i}$ in degree $j=id+d'$ and zero in all other degrees. By induction, 
$\beta_i(R/I_{s-1})$ is non zero only in degree $j=i+d'-1+(i-1)(d-1)=d'+id-d$. This shows that 
$\beta_i(R/I_{s})$ may be non zero only in degree $j=i+d'-1+(i-1)(d-1)$ and that 
$\beta_i(R/I_{s})=\beta_i(R/I_{s-1})+\beta_{i-1}(R/(I_{s-1}:x^{m_{s}})(-d'))$. The result now follows by induction 
and Theorem \ref{about quotients}. 
\end{proof}

\begin{exam}
An example of such ideal is $I=(abc,cde,cfg,chi)$ in the polynomial ring $k[a,b,c,d,e,f,g,h,i]$. The Betti numbers 
in homological degrees $1,2,3$ and $4$ are $4, 6, 4$ and $1$, and they lie in the degrees described in the corollary. 
It is easily seen that these are constructed from the Betti numbers of the colon ideals.
\end{exam}

\begin{cor}\label{lin.quot}
Let $I=(x^{m_1},\ldots,x^{m_t})$, $\deg x^{m_i}=d'$ for every $i=1,\ldots,t$, be a squarefree monomial ideal with 
linear quotients. If $I_{s-1}:x^{m_{s}}=(x_{s_1},\ldots,x_{s_{r_s}})$, $s=2,\ldots,t$, then for all 
$2\le i\le\mr{pd}_R(R/I)$, $\beta_{i,j}(R/I)$ is 
non zero only in degree $j=i+d'-1$ and we have
\[
\beta_{i,j}(R/I)=\sum_{s=2}^{t}\gf{r_s},{i-1}.
\]
\end{cor}

\bibliographystyle{amsplain}
\bibliography{ref}

\end{document}